\documentclass[12pt,reqno]{amsart}

\usepackage[usenames]{color}
\usepackage{amsmath,amssymb,pdfsync,verbatim,graphicx,epstopdf,enumerate,fancybox}
\usepackage{mathrsfs}
\usepackage{hyperref}
\usepackage{xcolor}
\hypersetup{colorlinks=true,linkcolor=blue}
\usepackage[notref,notcite]{}
\usepackage{textcomp}
\usepackage{manyfoot}
\usepackage{booktabs}
\usepackage{algorithm}
\usepackage{algorithmicx}
\usepackage{algpseudocode}
\usepackage{listings}

\numberwithin{equation}{section}
\everymath{\displaystyle}

\newtheorem{theorem}{Theorem}[section]
\newtheorem{lemma}[theorem]{Lemma}

\newtheorem{remark}[theorem]{Remark}

\newcommand*{\dis}{\displaystyle}
\newcommand*{\nn}{\nonumber}
\newcommand*{\R}{\mathbb{R}}
\newcommand*{\Rn}{\mathbb{R}^n}
\newcommand*{\Hn}{\mathbb{H}^n}
\newcommand*{\Bn}{\mathbb{B}^n}

\newcommand*{\divg}{\operatorname{div}}

\newcommand{\W}{\Omega}
\newcommand*{\dist}{\operatorname{dist}}
\newcommand*{\al}{\alpha_\lambda}
\newcommand*{\ve}{\varepsilon}

\newcommand*{\C}{\mathcal{C}}

\makeatletter
\@namedef{subjclassname@2020}{%
  \textup{2020} Mathematics Subject Classification}
\makeatother

\begin{document}
	
	\title{Symmetry for a quasilinear elliptic equation in Hyperbolic space}
	
	\begin{abstract}
		In this article we establish the sharp decay estimates and Hyperbolic symmetry of solutions to the quasilinear elliptic equation \begin{align*} &-\Delta_p^{\mathbb{H}^n} u - \lambda |u|^{p-2}u = |u|^{q - 2}u \text{ in } \mathbb{H}^n \\ & \quad u \ge 0, \; u \in \mathcal{D}^{1,p}(\mathbb{H}^n)\end{align*}
		in the $n$-dimensional real Hyperbolic space $\mathbb{H}^n$ where, $\lambda \in [0, \lambda_{\text{max}})$ with $\lambda_{\text{max}} = \left(\frac{n-1}{p}\right)^p$ the best constant of the Poincar\'e inequality in $\mathbb{H}^n$ and $p < q \le p^\ast$, $p^\ast = \frac{np}{n-p}$.
	\end{abstract}
	
	\author{Ramya Dutta$^\dagger$ and Kunnath Sandeep$^{\dagger\dagger}$}\thanks{$^{\dagger}$ Universit\'{e} Claude Bernard Lyon 1, 43 boulevard du 11 Novembre 1918, 69622 Villeurbanne cedex, France. E-mail : dutta@math.univ-lyon1.fr \\$^{\dagger\dagger}$ TIFR  Centre for Applicable Mathematics, Post Bag No. 6503, Sharadanagar, Yelahanka New Town, Bangalore 560065. Email: sandeep@tifrbng.res.in}
	
	\maketitle
	
\section{Introduction}
Understanding the sharp constants and extremals of Sobolev Inequalities is an important problem due to its applicability in the study of many partial differential equations. One of the important example in this direction is the well known Yamabe problem from differential geometry where the knowledge of extremals of the Euclidean Sobolev inequality or  more generally the classification of positive solutions of the corresponding Euler-Lagrange equation is used in a very crucial way.   In the case of Euclidean space extremals of Sobolev inequalities have been classified by now and one of the main development in the process has been establishing the radial symmetry of solutions. \\\\ In this article we study the Hyperbolic symmetry of  solutions of the following quasilinear PDE in the $n$ dimensional real Hyperbolic space $(\Hn,g)$ of constant sectional curvature $-1$:
\begin{align} \begin{cases} \label{R21eqn1}
			-\Delta_p^{\Hn} u - \lambda |u|^{p-2}u = |u|^{q - 2}u \text{ in } \Hn \\ u\ge 0, \; u \in \mathcal{D}^{1,p}(\Hn).
	\end{cases} \end{align} 
The equation \eqref{R21eqn1} is the Euler-Lagrange equation satisfied by the extremals of the Poincar\'{e}-Sobolev inequality in $(\Hn,g)$ given by 
\begin{align}\label{R2PSineq} \int_{\Hn} |\nabla_g u|_g^p - \lambda |u|^p \,dV_g \ge S_{\lambda,q} \left(\int_{\Hn} |u|^q \,dV_g\right)^{p/q}, \forall \, u \in C_c^1(\Hn) \end{align} where, $\lambda \in \left[0, \lambda_{\text{max}} \right)$ and $\dis \lambda_{\text{max}} := \inf_{u \in C_c^1(\Hn)\setminus 0} \frac{\int_{\Hn} |\nabla_g u|_g^p\,dV_g}{\int_{\Hn} |u|_g^p\,dV_g} = \left(\frac{n-1}{p}\right)^p$ denotes the best constant in the Poincar\'{e} inequality in Hyperbolic space,  $\Delta_p^{\Hn}u := \divg_g (|\nabla_g u|^{p-2}\nabla_g u)$ is the corresponding $p$-Laplace operator, $1 < p < n$ , $p < q \le p^\ast$ , $p^\ast = \frac{np}{n-p}$ denotes the critical Sobolev exponent and $\mathcal{D}^{1,p}(\Hn)$ is the completion of $C_c^1(\Hn)$ with the norm $\left(\int_{\Hn} |\nabla_g u|_g^p\,dV_g\right)^{1/p}$. Note that when $n=2$ we are only considering $p < 2$ and $q \in (2,p^\ast]$ where, $p^\ast = \frac{2p}{2-p}$. A solution of \eqref{R21eqn1} is a $u \in \mathcal{D}^{1,p}(\Hn)$ satisfying
$$  \int_{\Hn} \left[|\nabla_g u|_g^{p-2}g(\nabla_gu,\nabla_g\phi) - \lambda |u|^{p-2}u\phi \right] \,dV_g  = \int_{\Hn} |u|^{q-2}u\phi \,dV_g, \; \forall \, \phi \in \mathcal{D}^{1,p}(\Hn).$$
By writing this equation in either the Poincar\'{e} ball model or in the upper half-space model, it can be easily seen from the standard regularity theory (see \cite{Tol}, \cite{DiBe}) that the solutions of \eqref{R21eqn1} are indeed $C^{1,\alpha}$ for some $\alpha \in (0,1)$.
\\\\ Our main result in this article is the sharp asymptotics and Hyperbolic symmetry of solutions of \eqref{R21eqn1}.
    \begin{theorem}\label{R2hypsymmetry}
    	Let, $u$ be a solution of \eqref{R21eqn1}, then there is a point $O \in \Hn$ and a positive strictly monotone decreasing function $\Phi: [0,\infty) \to \mathbb{R}_+$ such that $u(x) = \Phi(\dist_{\Hn}(O,x))$. The function $\Phi(t)$ is $C^{1}$ at $t = 0$ with $\Phi'(0) = 0$ and smooth for all $t > 0$. It satisfies the pointwise bounds $ce^{-\al t} \le \Phi(t) \le Ce^{-\al t}$ for all $t \ge 0$, $ce^{-\al t} \le -\Phi'(t) \le Ce^{-\al t}$ for all $t \ge 1$, where, $c,C > 0$ are some positive constants (which depend on $\Phi$) and $\lim\limits_{t \to \infty} \frac{\Phi'(t)}{\Phi(t)} = - \al$, where, $\al$ is the unique positive root of $|\alpha|^{p-2}\alpha(n - 1 - (p-1)\alpha) = \lambda$ such that $\al \in \left(\frac{n-1}{p}, \frac{n-1}{p-1}\right]$.
    \end{theorem}
    \,
\\ In the case of Euclidean Sobolev inequality, the corresponding Euler-Lagrange equation 
\begin{align}\label{R2Eucl-Sob}
    -\Delta_p u  = |u|^{p^\ast - 2}u \text{ in } \Rn, u \ge 0, u \in \mathcal{D}^{1,p}(\Rn)
\end{align}
has been studied in detail. When $p=2$, the symmetry of solutions has been established in the seminal work of Gidas-Ni-Nirenberg \cite{GNN} using the moving plane method. However when $p\not= 2$ there are many obstructions for proving the symmetry of solutions for the $p$-Laplace equations using the moving plane method mainly due to the lack of strong comparison principle, lack of tools to get the moving plane argument started from infinity (in the case of problems in unbounded domain) etc.  Note that the $p$-Laplace operator is elliptic on points where $\nabla u \not=0$ and in the complement it is either degenerate or singular depending on $p > 2$ or $p < 2$ respectively. The first significant progress in applying the moving plane method for the $p$-Laplace case was made by Damascelli in \cite{Dam} where 
 he proved the radial symmetry for positive solutions of equations of the form $-\Delta_pu = f(u)$ with Dirichlet boundary condition in bounded domains in the Euclidean space. The results in the entire space $\mathbb{R}^n$ was established in \cite{DPM} for solutions of the same PDE with proper decay/integrability, where $f$ was non increasing in some interval $(0, s_0)$. See also \cite{SH} for related developments. \\\\
 In the case of Sobolev extremals, the nonlinear term is strictly increasing and hence the results of \cite{DPM} can not be applied to obtain radial symmetry of Sobolev extremals. It was shown in 
 \cite{DM} that when $1 < p < 2$ the symmetry of \eqref{R2Eucl-Sob} follows if the solutions have a decay namely $u(x) \le C |x|^{-m}$, $|\nabla u| \le C|x|^{-(m+1)}$ for $|x|$ large (additionally $u(x) \ge c|x|^{-m}$ for $|x|$ large when $p^\ast < 2$) where, $m > \frac{n-p}{p}$. The symmetry result was proved in \cite{DMMS} without any additional decay assumption when $\frac{2n}{n+2} \le p < 2$ (i.e., $p < 2$ and $p^\ast \ge 2$). \\\\
Tools for handling the issue of strong comparison in the general case $p > 2$ was developed by Damascelli and Sciunzi in \cite{DS1}, \cite{DS2}. Again these informations are not enough to obtain symmetry results for \eqref{R2Eucl-Sob} in the full range of $1 < p < n$. The first break through in this direction was obtained by V\'{e}tois \cite{Vetois} where he proved that the solutions of \eqref{R2Eucl-Sob} satisfy the sharp decay estimate $u(x) \approx |x|^{-\frac{n-p}{p-1}}$ and $|\nabla u| \le C|x|^{-\frac{n-1}{p-1}}$ for $|x|$ large. With this and the results of \cite{DM} the symmetry results follows when $1 < p < 2$. The symmetry result for the case $p>2$ was completed by Sciunzi \cite{Sciunzi} by establishing a precise pointwise gradient bound 
 $|\nabla u(x)| \approx |x|^{-\frac{n-1}{p-1}}$ for $|x|$ large. More recently the classification result for \eqref{R2Eucl-Sob} has also been established by a technique independent of the moving plane method in conical domains as well as the full Euclidean space $\Rn$ in \cite{CFR}. We also refer to \cite{catino2022}, \cite{ou2022} and \cite{vetois2024} for the classification result of \eqref{R2Eucl-Sob} where the finite energy assumption $u \in \mathcal{D}^{1,p}(\Rn)$ has been relaxed in certain range of $p$ and $n$. \\\\
 Coming to the equation \eqref{R21eqn1}, when $p=2$ the symmetry of solution has been studied in \cite{ADG} for $\lambda \le \frac14 n(n-2)$ and  subsequently extended to the full range of $\lambda$ in \cite{MS} using the method of moving plane. Consequently in \cite{MS} the solutions have been studied thoroughly by analysing the corresponding ordinary differential equation. In this situation the moving plane method can be applied as we are in the elliptic regime and we have appropriate Sobolev inequalities. Thus the Euclidean method can be adopted  to this geometric situation by reflecting over foliations of the Hyperbolic space. \\\\
Our main focus in this article is the case $p \not= 2 $. Like in the Euclidean case here also the issues are lack of strong comparison principles and also moving the  plane from infinity.  The issue of strong comparison principle in bounded subsets of $\Hn$ can be established by suitably adapting the ideas of \cite{DS1}, \cite{DS2} to the Hyperbolic space, the main challenge in this case is to establish the precise asymptotic estimates at infinity. \\\\
The most important contribution of this article is to establish sharp asymptotic estimates to the positive solutions of \eqref{R21eqn1} and its gradient, so that the moving plane method can be made to work in this case by suitably adapting the Euclidean tools. In the Euclidean case the precise $L^\infty$-estimates were obtained by crucially rescaling the equations. However in our case we do not have any substitute for scaling. In Theorem-\ref{R2sharpdecay} we will establish the $L^\infty$-estimates by finding appropriate sub and super solutions combined with weak comparison tools using the Picone identity. The estimates we obtained in Theorem-\ref{R2sharpdecay} are new even in the case of $p=2$, which in fact improves the estimates obtained in \cite{MS} except in cases where explicit solutions are known. To prove the estimates on the gradient, the upper bound follows from the $L^\infty$-estimate and a use of Harnack inequality, however the lower bound on the gradient at infinity is the difficult estimate to prove. In the Euclidean case \eqref{R2Eucl-Sob} this is achieved in \cite{Sciunzi} by a blow up analysis via rescaling which leads to the limiting equation $-\Delta_p u = 0$ in $\mathbb{R}^n \setminus \{0\}$ and $u(x) \approx |x|^{-\frac{n-p}{p-1}}$ whose solutions have been classified in \cite{kv} to be constant multiples of the fundamental solution $u(x) = c|x|^{-\frac{n-p}{p-1}}$ (see also \cite{Sciunzi} for an alternative proof under weaker assumptions). Also see \cite{OSV} for a similar limiting equation with a Hardy potential with singularity at origin in $\Rn$. The main difference of our work when comparing with the Euclidean references mentioned above is in the nature of the blow up and the limiting problem. In our case the limiting problem is in fact an eigenvalue problem with a pole at `infinity'.  We need to establish a classification result for positive eigenfunctions of the following eigenvalue problem
\begin{equation}\label{R2eigenfunction}
-\Delta_p^{\Hn} u =  \lambda u^{p-1}
\end{equation}
which are comparable to certain eigenfunctions with desired behavior at infinity. Here by an eigenfunction we means a $C^1$ function which satisfy the relation 
$$\int_{\Hn} |\nabla_g u|_g^{p-2}g(\nabla_gu,\nabla_g\phi)\,dV_g = \int_{\Hn} \lambda u^{p-1} \phi \,dV_g$$ for all $\phi \in C_c^1(\Hn)$. \\\\
Fix a point $O\in \Hn$ and $(r,\theta) \in (0,\infty)\times \mathbb{S}^{n-1}$ be the geodesic normal coordinates centred at $O$. For $\xi \in \mathbb{S}^{n-1}$ define a function $E_\xi$ by
\begin{align}\label{R2eigen0}
	E_\xi(x) = \left( \cosh r -(\xi \cdot \theta)  \sinh r \right)^{-\alpha_\lambda}
\end{align} 
where, $(r,\theta)$ is the geodesic normal coordinate of $x$ and $\al$ is as in Theorem-\ref{R2hypsymmetry}. Note that $E_\xi$ extends as a $C^1$ function to $\Hn$. With this definition we can state the classification result as follows.
\begin{theorem}\label{R2eigenclfn} 
	For any $\xi \in \mathbb{S}^{n-1}$, $E_\xi$ be as in \eqref{R2eigen0}. If  $E$ satisfies \eqref{R2eigenfunction} and has the bounds $E_\xi (x)\le E(x) \le CE_\xi (x)$ for some $C > 1$ and all $x \in \Hn$ then $E \equiv cE_\xi$ for some $c \in [1,C]$.
\end{theorem}
\, \\ We prove Theorem-\ref{R2eigenclfn} in Section-\ref{R2decay}, Theorem-\ref{R2classeigen} formulated in the equivalent Poincar\'{e} ball model $\Bn$ of the Hyperbolic space.
\\ The problem of classification of positive eigenfunctions in general may have independent interest from a potential theoretic perspective, particularly in the study of Martin boundary. We recall some relevant definitions. Let, $\mathcal{L}$ be a homogeneous elliptic operator, meaning if $u$ solves $\mathcal{L}u = 0$ then so does $cu$ for any constant $c$. Suppose, $u \in C^1(\Hn)$ be a positive solution to the equation $\mathcal{L}u = 0$ such that $u$ has vanishing `trace' on $\partial_\infty \Hn \setminus \{\xi\}$ i.e., it vanishes in some generalized trace sense on the the ideal boundary (or sphere at infinity) except at $\xi$ (see the definition of ideal boundary or the sphere at infinity in Section-\ref{R2prelim}). We call such a solution $u$ to be a minimal function relative to the point $\xi \in \partial_\infty \Hn$ (identified with $\mathbb{S}^{n-1}$) if and only if any other solution $\tilde{u}$ of $\mathcal{L}\tilde{u} = 0$ satisfying $0 \le \tilde{u} \le Cu$ for some positive constant $C$, must be a positive constant multiple of $u$, i.e., $\tilde{u} \equiv cu$ for some $c > 0$. We say that the $\mathcal{L}$-Martin boundary can be identified with $\partial_\infty \Hn$ if each point $\xi \in \partial_\infty \Hn$ has a unique (up to constant multiple) positive minimal function relative to it.
\\\\ In the linear case when $p = 2$, the conclusion of Theorem-\ref{R2eigenclfn} can be obtained by only assuming the upper bound (see Remark-\ref{R2rmkminimal}). The minimality of the eigenfunctions $E_\xi$ relative to $\xi \in \mathbb{S}^{n-1}$ can be seen as a direct consequence of a representation theorem based on its boundary behavior due to Helgason \cite{Helgason} and Minemura \cite{Minemura}. The representation of eigenfunctions on non-compact symmetric space through a hyperfunction on its minimal boundary, known as the Helgason conjecture, was established in full generality in \cite{kami}. It would be interesting to know if the eigenfunctions $E_\xi$ are minimal relative to $\xi \in \mathbb{S}^{n-1}$ for a general $1 < p < n$, i.e., if the assumption of the lower bound in Theorem-\ref{R2eigenclfn} can be relaxed.
\\\\ The structure of the article is as follows. In Section-\ref{R2prelim} we go over some of the basic notations and definitions relating to the Hyperbolic space. In Section-\ref{R2decay} we prove the boundedness, sharp decay estimates for the solutions of \eqref{R21eqn1} and its gradient, as well as the classification of eigenfunctions Theorem-\ref{R2eigenclfn} (Theorem-\ref{R2classeigen}). Section-\ref{R2symmetry} is dedicated to the proof of Hyperbolic symmetry of \eqref{R21eqn1} and in Section-\ref{R2strongcopmarison} we discuss the Sobolev regularity of $u$ and the strong comparison theorem in Hyperbolic space. Finally in Section-\ref{R2exist} we discuss some existence and non-existence results related to \eqref{R21eqn1}.

\section{preliminaries}\label{R2prelim} In this section we will recall some of the basic notations and results which we will be using in this article.\\
Notations : We will denote the $n$ dimensional real Hyperbolic space by $(\Hn,g)$. Fix a point $O\in \Hn$ then in the normal coordinates based at $O$ we will identify the unit sphere in the tangent space of $\Hn$ at $O$ by $\mathbb{S}^{n-1}$.  Also note that for $\Hn$ we have the notion of the geometric boundary called the `sphere at infinity' or the `ideal boundary' denoted $\partial_\infty \Hn$ which is abstractly defined as the equivalence classes of geodesic rays under the equivalence relation where two geodesic rays $\gamma_1, \gamma_2$ are equivalent if $\dist_g(\gamma_1(t), \gamma_2(t))$ remains bounded for $t\ge 0.$ Once we fix the point $O\in \Hn$, we can identify this sphere at infinity (or the ideal boundary) with $\mathbb{S}^{n-1}$. We refer the book \cite{SY-Book}, Chapter-2 for details.\\
In geodesic normal coordinates the Hyperbolic metric takes the form $ g = dt^2 + (\sinh t)^2 g_{\mathbb{S}^{n-1}}$ where, $t(x) = \dist_g(O,x)$ and  $g_{\mathbb{S}^{n-1}}$ is the standard metric on $\mathbb{S}^{n-1}$. Consequently the Hyperbolic measure will be given by $dV_g = (\sinh t)^{n-1}dV_{g_{\mathbb{S}^{n-1}}}dt$. \\\\
When we need to make precise calculations we will be working with either the Poincar\'{e} ball model or the upper half-space model. In both the models we will be using $g$ for the metric which will vary depending on the model we are using. We will denote the gradient, divergence and modulus with respect to the metric by $\nabla_g, \divg_g$ and $|\cdot |_g$ respectively. While working in the Poincar\'{e} ball or upper half-space models the corresponding  quantities with respect to the Euclidean metric will be denoted by $\nabla, \divg$ and $|\cdot |$ respectively. Since the gradient and divergence are intrinsic quantities (see \cite{canzani2013}, Chapter-4), we see that our equation \eqref{R21eqn1} is invariant under isometries of the Hyperbolic space, i.e., if $u$ solves \eqref{R21eqn1} and $T$ is an isometry of $\Hn$ then $u\circ T$ also solves \eqref{R21eqn1}. \\\\
 We will denote the Poincar\'{e} ball model by $(\Bn,g)$ where, $\Bn$ is the Euclidean unit ball in $\Rn$ endowed with the metric $\left(\frac{2}{1-|x|^2} \right)^2g_E$ and $g_E$ denotes the Euclidean metric. By an abuse of notation, we will denote the upper half-space model by $(\Hn,g)$  where, $\Hn = \{x \in \Rn : x_n >0\}$ and the metric $g$ is given by $g = x_n^{-2} g_E$.  Note that the isometry group in these two models are respectively the  Mobius groups of the ball and upper half-space which we will denote by $\mathcal{M}(\Bn)$ and $\mathcal{M}(\Hn)$. For more details about the two models and their isometry groups we refer to \cite{stoll} and \cite{Ratcli}. Throughout the article unless otherwise stated $B_R(x_0)$ will denote the metric ball centered at $x_0$ with radius $R$, meaning $B_R(x_0) := \left\{x \in \Hn: \dist_{g}(x,x_0) < R\right\}$.\\
We will end this section by recalling the expressions for the gradient, divergence and the p-Laplacian in these coordinates. In the Poincar\'{e} ball model $(\Bn,g)$ they are given by 
\begin{align} \nabla_g u(x) &= \left(\frac{1 - |x|^2}{2}\right)^2\nabla u(x) \end{align}
\begin{align} |\nabla_g u|_g(x) &= \left(\frac{1 - |x|^2}{2}\right)|\nabla u(x)| \end{align}
\begin{align} \divg_g X &= \left(\frac{1 - |x|^2}{2}\right)^n\divg \left(\left(\frac{2}{1 - |x|^2}\right)^n X\right). \end{align}
The p-Laplacian is given by 
\begin{align}
	\Delta_p^{\Bn} u(x) = \divg_g(|\nabla_g u|_g^{p-2}\nabla_g u) = \left( \frac{1-|x|^2}{2}\right)^n \divg \left(\left( \frac{2}{1-|x|^2}\right)^{n-p}|\nabla u|^{p-2}\nabla u\right).
\end{align}
In the upper half-space model $(\Hn, g)$ they are given by
\begin{align}
	\nabla_g u(x) &= x_n^2\nabla u(x) \\ |\nabla_g u|_g(x) &= x_n|\nabla u(x)| \\ \divg_g X &= x_n^n\divg \left(x_n^{-n} X\right).
\end{align}
The p-Laplacian is given by
\begin{align}
	\Delta_p^{\Hn} u(x) = \divg_g(|\nabla_g u|_g^{p-2}\nabla_g u)= x_n^n \divg \left(x_n^{p-n}|\nabla u|^{p-2}\nabla u \right).
\end{align}
 
	\section{Decay estimate for solutions and its gradient}\label{R2decay}
	We begin by showing that the solutions to equation \eqref{R21eqn1} have higher integrability and are bounded by proceeding along the arguments in \cite{Vetois}, \cite{Vetois2}. 
	\\ Let us consider the auxiliary function 
	\begin{align}\label{R2auxf}
		f(\alpha) := |\alpha|^{p-2}\alpha(n-1 - (p-1)\alpha)
	\end{align} 
	then, $\displaystyle f(0) = f\left(\frac{n-1}{p-1}\right) = 0$ and $f'(\alpha) = (p-1)|\alpha|^{p-2}(n-1 - p\alpha)$. Thus $f$ has a unique maximum at $\displaystyle \alpha_c = \frac{n-1}{p}$ and is concave and increasing in $[0,\alpha_c)$ and decreasing in $\left(\alpha_c,\dfrac{n-1}{p-1}\right]$. Thus the equation $f(\alpha) = \lambda$ has two roots $\alpha_\lambda, \beta_\lambda$ satisfying \begin{align}\label{R2root} 0 \le \beta_\lambda < \frac{n-1}{p} < \alpha_\lambda \le \frac{n-1}{p-1} \end{align} for any $\displaystyle \lambda \in \left[0, \left(\frac{n-1}{p}\right)^p\right)$.
	\subsection{Boundedness of solution}\,
	\\ In what follows we will use the notation $\W_m = \left\{x \in \Bn : |u| > m \right\}$, its characteristic function $\chi_m := \chi_{\W_m}$ and the characteristic function of complement will be denoted by $\chi_m^c := \chi_{\Bn \setminus \W_m}$.
	
	\begin{lemma}\label{R2hint} 
		Let $u$ be a solution to \eqref{R21eqn1} when $q = p^\ast$. Then $u \in L^k(\Bn)$ for $p \le k < p^\ast \left(\frac{n-1}{p\beta_\lambda}\right)$ where, $\beta_{\lambda}$ is as in \eqref{R2root}. Hence, $u$ has higher integrability than $p^\ast$.
	\end{lemma}
	
	\begin{proof} Note that since $u \in \mathcal{D}^{1,p}(\Bn)$ by the Poincar\'{e}-Sobolev inequality in $\Bn$ we know that $u \in L^k(\Bn)$ for $p \le k \le p^\ast$. Let us denote by $u_m := \min(|u|,m)$ for $m > 0$. Testing equation \eqref{R21eqn1} with $\varphi = u_m^{p\beta - p}u \in \mathcal{D}^{1,p}(\Bn)$ where $\beta > 1$, we have 
		\begin{align}\label{R2beqn1} \int_{\Bn} |\nabla_g u|_g^{p-2}g(\nabla_g u, \nabla_g \varphi) \,dV_g - \lambda \int_{\Bn} |u|^{p-2}u\varphi \,dV_g = \int_{\Bn} |u|^{q-2}u\varphi \,dV_g. \end{align} Then on the LHS of \eqref{R2beqn1} we have 
		\begin{align}\label{R2beqn2}
			\int_{\Bn} |\nabla_g u|_g^{p-2}g(\nabla_g u, \nabla_g \varphi) \,dV_g &= \int_{\Bn} |\nabla_g u|_g^{p} \left( \chi_m^c (p\beta - p + 1)|u|^{p\beta - p} + \chi_m m^{p\beta - p} \right) \,dV_g \nn \\ & \ge \frac{(p\beta - p + 1)}{\beta^p}\int_{\Bn} |\nabla_g u|_g^{p} \left( \chi_m^c \beta^p |u|^{p\beta - p} + \chi_m m^{p\beta - p} \right) \,dV_g \nn \\& = \frac{(p\beta - p + 1)}{\beta^p} \int_{\Bn} \left|\nabla_g \left(u_m^{\beta - 1}u\right)\right|_g^{p} \,dV_g
		\end{align} where, we used the Bernoulli's inequality $\beta^p = (1 + (\beta - 1))^p \ge 1 + p(\beta - 1)$. Again using Poincar\'{e} inequality in $\Bn$ we have 
		\begin{align}\label{R2beqn3} \left(\frac{p\beta - p + 1}{\beta^p} - \frac{\lambda}{\lambda_{\text{max}}}\right)\int_{\Bn} \left|\nabla_g \left(u_m^{\beta - 1}u\right)\right|_g^{p} \,dV_g \le \int_{\Bn} |u|^qu_m^{p\beta - p} \,dV_g. \end{align}
		Therefore, on the RHS of \eqref{R2beqn3} for any $m > m_0$ we have 
		\begin{align} \label{R2beqn4} \int_{\Bn} |u|^{q}u_{m}^{p\beta - p} \,dV_g \le m_0^{p\beta - p}\int_{\W_{m_0}^c} |u|^{q} \,dV_g + \int_{\W_{m_0}} |u|^q u_m^{p\beta - p} \,dV_g. \end{align}
		Using Holder's inequality followed by Sobolev Inequality on the second term on RHS of \eqref{R2beqn4} we get 
		\begin{align}\label{R2beqn5} \int_{\W_{m_0}} |u|^q u_m^{p\beta - p} \,dV_g &\le \left(\int_{\W_{m_0}} |u|^{q} \,dV_g\right)^{1 - \frac{p}{q}}\left(\int_{\W_{m_0}} \left(u_m^{\beta - 1}|u|\right)^{q} \,dV_g\right)^{\frac{p}{q}} \nn \\&\le S\left(\int_{\W_{m_0}} |u|^{q} \,dV_g\right)^{1 - \frac{p}{q}} \int_{\Bn} \left|\nabla_g \left(u_m^{\beta - 1} u\right) \right|_g^p \,dV_g \nn \\& \le \frac{1}{2}\left(\frac{p\beta - p + 1}{\beta^p} - \frac{\lambda}{\lambda_{\text{max}}}\right) \int_{\Bn} \left|\nabla_g \left(u_m^{\beta - 1} u\right) \right|_g^p \,dV_g \end{align} by choosing $m_0$ sufficiently large depending on $\beta$. Therefore, combining \eqref{R2beqn3}, \eqref{R2beqn4} and \eqref{R2beqn5} we get
		\begin{align} \label{R2beqn6} \left(\frac{p\beta - p + 1}{\beta^p} - \frac{\lambda}{\lambda_{\text{max}}}\right) \int_{\Bn} \left|\nabla_g \left(u_m^{\beta - 1} u\right) \right|_g^p \,dV_g \le 2m_0^{p\beta - p} \int_{\W_{m_0}^c} |u|^{q} \,dV_g \end{align} followed by an application of Sobolev inequality and letting $m \to \infty$ we get 
		\begin{align}\label{R2beqn7} \left(\frac{p\beta - p + 1}{\beta^p} - \frac{\lambda}{\lambda_{\text{max}}}\right)\left(\int_{\Bn} |u|^{\beta q} \,dV_g\right)^{\frac{p}{q}} \le Cm_0^{p\beta - p} \int_{\W_{m_0}^c} |u|^{q} \,dV_g. \end{align} Now, note that making the change of variable $\dis \beta = \frac{n-1}{p\alpha}$ we have $$\frac{p\beta - p + 1}{\beta^p} = \frac{\alpha^{p-1}(n - 1 - (p-1)\alpha)}{\lambda_{\text{max}}} > \frac{\lambda}{\lambda_{\text{max}}}$$ provided $\alpha \in (\beta_\lambda, \alpha_\lambda)$ as in \eqref{R2root} i.e., if $\beta \in \left(\frac{n-1}{p\al}, \frac{n-1}{p\beta_\lambda}\right)$. This concludes the proof of higher integrability. \end{proof}
	
	\begin{theorem}\label{R2lemmavanish} Let $u$ be a solution of \eqref{R21eqn1} then $u \in L^\infty(\Bn)$. Furthermore, $u(x) \to 0$ and $|\nabla_g u|_g(x) \to 0$ as $\dist_{\Bn}(0,x) \to \infty$. \end{theorem}
	
	\begin{proof} The proof in the sub-critical case $q < p^\ast$ follows in a similar way to the critical case. Note that in the critical case $q = p^\ast$ we may rewrite equation \eqref{R21eqn1} as \begin{align}\label{R2beqn8} -\Delta_p^{\Bn} u = (\lambda + |u|^{p^\ast-p})|u|^{p-1}u \end{align} where, $\lambda + |u|^{p^\ast-p} \in L^k_{\text{loc}}(\Bn)$ for some $\dis k > \frac{n}{p}$ by the higher integrability proved in Lemma-\ref{R2hint}.
		\\ Therefore, by usual Moser's iteration (see \cite{serrin65}, \cite{T}) we have the estimate \begin{align} \label{R2beqn9} \sup_{B_{1}(0)} |u| \le C\lVert u \rVert_{L^{p^\ast}(B_2(0))} \end{align} where, the constant $C$ only depends on $\lVert\lambda + |u|^{p^\ast-p} \rVert_{L^{k}(B_2)} \le C_1\left(\lambda + \lVert u \rVert_{L^{p^\ast\beta_0}(\Bn)}^{p^\ast-p}\right)$ for $\dis 1 < \beta_0 < \frac{n-1}{p\beta_\lambda}$. Since, the equation \eqref{R2beqn8} is invariant under isometries of the Hyperbolic space $\mathcal{M}(\Bn)$ then for any $x \in \Bn$ consider $\tau_x \in \mathcal{M}(\Bn)$ such that $\tau_x(0) = x$. Then using \eqref{R2beqn9} on $v := (u \circ \tau_x)$ we have $\dis \sup_{B_{1}(0)} |v| \le C\lVert v \rVert_{L^{p^\ast}(B_2(0))}$ or equivalently $\dis \sup_{B_{1}(x)} |u| \le C\lVert u \rVert_{L^{p^\ast}(B_2(x))}$ for any $x \in \Bn$. Combining this with standard $C^{1,\sigma}$ estimates (see \cite{DiBe}, \cite{Tol}) we have $\dis \sup_{B_{1/2}(x)} |\nabla_g u|_g \le C'\sup_{B_{1}(x)} |u| \le C''\lVert u \rVert_{L^{p^\ast}(B_2(x))}$ for all $x \in \Bn$ and the later goes to $0$ as $\dist_{\Bn}(0,x) \to \infty$ since $u \in L^{p^\ast}(\Bn)$. \end{proof}
	
	\subsection{Super-solution and sub-solutions to the equation}\, \\
	In the following lemma we identify the candidates for sub(super)-solutions to our equation which will be used to establish the decay estimates.
	
	\begin{lemma}\label{R2lemmasubsol} The function $\displaystyle v(x) := \left[\cosh \left(\frac{\dist_{\Bn}(0,x)}{m}\right)\right]^{-m\al}$ for $m \ge 2$ is a super-solution to the equation \begin{align}\label{R2supersol} -\Delta_p^{\Bn} v - \lambda v^{p-1} \ge c_\lambda e^{-\frac{2}{m}\dist_{\Bn}(0,x)}v^{p-1} \end{align} in $\displaystyle \dist_{\Bn}(0,x) > R_\lambda$ for some $R_\lambda > 0$ and some positive constant $c_\lambda > 0$.
	\\ Also, $\displaystyle v(x) := \left[\sinh \left(\frac{\dist_{\Bn}(0,x)}{2}\right)\right]^{-2\al}$ is a sub-solution to the equation \begin{align}\label{R2subsol} -\Delta_p^{\Bn} v - \lambda v^{p-1} \le 0 \end{align} for $\displaystyle \dist_{\Bn}(0,x) > R_\lambda$ for some $R_\lambda > 0$. \end{lemma}
	\begin{proof} Recall that in $\Bn$ we have $\dis t_x := \dist_{\Bn}(0,x) = \log\left(\frac{1+|x|}{1 - |x|}\right)$ or equivalently $r_x := |x| = \tanh(t_x/2)$. Therefore we may express the Riemannian metric $g$ in geodesic polar coordinates at the origin as 
	\begin{align}\label{R2polar}  
		g = \frac{4dx^2}{(1 - |x|^2)^2} = \left(\frac{2}{1 - r^2}\right)^2(dr^2 + r^2g_{\mathbb{S}^{n-1}}) = dt^2 + (\sinh t)^2 g_{\mathbb{S}^{n-1}}. 
	\end{align} 
	Then the $\Delta_p^{\Bn}$ of a radial function $v(x)$, which we continue to denote as $v(t)$ with $t = \dist_{\Bn}(0,x)$, is given by 
	\begin{align}\label{R2polarplaplace} 
		\Delta_p^{\Bn} v(x) &= \frac{1}{(\sinh t)^{n-1}} \left((\sinh t)^{n-1} |v'(t)|^{p-2}v'(t)\right)' \nn \\&= |v'(t)|^{p-2}\left[(p-1)v''(t) + (n-1)\coth (t) v'(t)\right]. 
	\end{align} 
	\\ We consider the function $\displaystyle v(t) := \left[\cosh \left(\frac{t}{m}\right)\right]^{-m\al}$ where, $m \ge 2$. Then by a direct computation of the derivatives we have $$v'(t) = -\al \left[\cosh \left(\frac{t}{m}\right)\right]^{-m\al}\tanh \left(\frac{t}{m}\right) = -\al \tanh \left(\frac{t}{m}\right) v(t)$$ and $$v''(t) = -\frac{\al}{m} \left[\cosh\left(\frac{t}{m}\right)\right]^{-2}v(t) + \al^2 \tanh^2\left(\frac{t}{m}\right) v(t).$$ Substituting in the expression \eqref{R2polarplaplace} we get 
	\begin{align}\label{R22eq3} 
		-\Delta_p^{\Bn} v &= \al^{p-2}\tanh^{p} \left(\frac{t}{m}\right)\left[ (p-1)\frac{\al}{m}\left[\sinh\left(\frac{t}{m}\right)\right]^{-2} - (p-1)\al^2 \right. \nn \\ & \quad \left. + (n-1)\al \coth t \coth\left(\frac{t}{m}\right) \right]v^{p-1}(t). 
	\end{align} 
	Now we have the asymptotic expansions of the following functions and writing $e^{-t/m} = s$ for brevity 
	\begin{align} \left[\sinh\left(\frac{t}{m}\right)\right]^{-2} &= 4e^{-2t/m}\left(1 + 2 e^{-2t/m} + O(e^{-4t/m})\right) = 4s^2 + O(s^4) \nn
		    \\ \coth \left(\frac{t}{m}\right) &= 1 + 2e^{-2t/m} + O(e^{-4t/m}) = 1 + 2s^2 + O(s^4) \nn
			\\  \coth t &= 1 + 2e^{-2t} + O(e^{-4t}) = 1 + 2s^{2m} + O(s^{4m}) \nn
			\\  \tanh^{p} \left(\frac{t}{m}\right) &= \left(1 - 2e^{-2t/m} + O(e^{-4t/m})\right)^p = 1 - 2pe^{-2t/m} + O(e^{-4t/m}) = 1 - 2ps^2 + O(s^4). 
	\end{align}
		\\ Substituting the asymptotic expressions in \eqref{R22eq3} we get 
		\begin{align} \label{R22eq4} 
			-\Delta_p^{\Bn} v &= \al^{p-2} (1 - 2ps^2 + O(s^4))\left[(n-1)\al (1+2s^2+O(s^4))(1 + 2s^{2m} + O(s^{4m})) \right. \nn \\ & \quad \left. - (p-1) \al^2 + 4(p-1)\frac{\al}{m}(s^2 + O(s^4))\right]v^{p-1}(t) \\ &= (1 - 2ps^2 + O(s^4)) \left[\underbrace{\al^{p-1}(n-1 - (p-1)\al)}_{ =: f(\al) = \lambda} + 2\al^{p-1}(n-1)s^2 \right. \nn \\ & \qquad\qquad\qquad\qquad\qquad \left. + \frac{4\al^{p-1}}{m}(p-1)s^2 + O(s^4)\right]v^{p-1}(t) \\&= \left[\lambda + 2\left(\al^{p-1}(n-1) - p\lambda\right) s^2 + \frac{4\al^{p-1}}{m}(p-1)s^2 + O(s^4) \right]v^{p-1}(t) \\ &= \lambda v^{p-1}(t) + 2c_{\lambda} e^{-2t/m}\left(1 + O(e^{-2t/m})\right)v^{p-1}(t) 
		\end{align} 
	where, $f(\alpha)$ is as in \eqref{R2auxf}, the constant $\displaystyle c_\lambda := \left(\al^{p-1}(n-1) - p\lambda\right) + \frac{2\al^{p-1}}{m}(p-1)$ is positive since, $\displaystyle \al > \frac{n-1}{p}$ and $\displaystyle \lambda < \left(\frac{n-1}{p}\right)^p$. Therefore the function $\displaystyle v(x) := \left[\cosh \left(\frac{\dist_{\Bn}(0,x)}{m}\right)\right]^{-m\al}$ for $m \ge 2$ is a super-solution to the equation \eqref{R2supersol} whenever $t = \dist_{\Bn}(0,x) > R_\lambda$ is chosen large enough.
	\\ By a similar calculation we also have $\displaystyle v(x) := \left[\sinh \left(\frac{\dist_{\Bn}(0,x)}{2}\right)\right]^{-2\al}$ is a sub-solution to \eqref{R2subsol} in $\dist_{\Bn}(0,x) > R_\lambda$. \end{proof} 
	\subsection{Proof of sharp decay} \,
	\\ We first prove a weaker decay estimate for solutions to \eqref{R21eqn1} which we will subsequently reuse to prove the sharp decay estimate. Before proceeding with the proof of decay estimates we recall the following Picone-inequality: \begin{lemma}{(Picone Inequality)} \label{R2piconelemma} Let $p \in (1,n)$ and $u,v \in W^{1,p}_{\text{loc}}(\Bn)$ be two positive functions then 
		\begin{align}\label{R2picone} 
			\mathscr{B}(u,v) := & |\nabla_g u|_g^{p-2} g\left( \nabla_g u , \nabla_g\left(u - \frac{v^p}{u^p} u\right) \right) + |\nabla_g v|_g^{p-2} g\left( \nabla_g v , \nabla_g\left(v - \frac{u^p}{v^p} v\right) \right) \nn \\ &\ge C_p \min\left\{u^p, v^p\right\}\left(|\nabla_g \log u|_g + |\nabla_g \log v|_g\right)^{p-2}|\nabla_g (\log u - \log v)|_g^2. 
	\end{align} \end{lemma}
	\begin{proof}
		Directly follows from the $p$-homogeneity and multiplying by the conformal factor in the Euclidean inequality from Lemma-3.1 in \cite{Xiang} when $1 < p < 2$ and Lemma-3.1 of \cite{OSV} when $p > 2$.
	\end{proof}
	\begin{lemma}\label{R2weakdecay} Let, $u$ be a solution to \eqref{R21eqn1}. Then for each $\ve > 0$ such that $\lambda + \ve < \lambda_{\text{max}}$ there is a constant $C_{\lambda + \ve} > 0$ such that \begin{align}\label{R2weak1} u(x) \le C_{\lambda+\ve} \left[\cosh \left(\frac{\dist_{\Bn}(0,x)}{2}\right)\right]^{-2\alpha_{\lambda + \ve}}, \, \forall \, x \in \Bn. \end{align} \end{lemma}
	\begin{proof} 
		Recall from Lemma-\ref{R2lemmasubsol}, the function $\dis v_{\lambda+\ve}(x) := C_{\lambda + \ve}\left[\cosh \left(\frac{\dist_{\Bn}(0,x)}{2}\right)\right]^{-2\alpha_{\lambda + \ve}} \in \mathcal{D}^{1,p}(\Bn)$ is a super-solution to the equation 
		\begin{align}\label{R2eqw1} 
			-\Delta_p^{\Bn} v_{\lambda + \ve} - \lambda v_{\lambda + \ve}^{p-1} \ge \ve v_{\lambda + \ve}^{p-1}. 
		\end{align} 
		Using Lemma-\ref{R2lemmavanish} we may choose $R_{\lambda + \ve} > 0$ large enough such that $u^{q - p}(x) \le \ve$ in $\Bn \setminus B_{R_{\lambda + \ve}}$. Then $u$ is a subsolution to the equation 
		\begin{align}\label{R2eqw2} 
			-\Delta_p^{\Bn} u - \lambda u^{p-1} \le \ve u^{p-1}. 
		\end{align} 
		Finally we may fix $C_{\lambda + \ve} > 0$ large enough such that $u \le v_{\lambda + \ve}$ in $B_{R_{\lambda + \ve}}$. Let $\eta \in C_c^\infty(\Bn)$ be a non-negative test function such that $\eta \equiv 1$ in $B_R$ and $\eta \equiv 0$ in $\Bn \setminus B_{R+1}$ where, $R > R_{\lambda + \ve}$. Then testing \eqref{R2eqw2} with $\varphi_1 := \eta u^{1-p}(u^p - v_{\lambda+\ve}^p)_+$ and \eqref{R2eqw1} with $\varphi_2 := \eta v_{\lambda+\ve}^{1-p}(u^p - v_{\lambda+\ve}^p)_+$ which are admissible test functions supported in $B_{R+1} \setminus B_{R_{\lambda + \ve}}$ and subtracting we get 
		\begin{align}\label{R2eqw3} 
			& \int_{\Bn \cap \left\{u \ge v_{\lambda + \ve}\right\}} \eta \mathscr{B}\left(u,v_{\lambda + \ve}\right) \,dV_g \nn \\ &\le \int_{\Bn \cap \left\{u \ge v_{\lambda + \ve}\right\}} \eta (u^{q - p} - \ve)(u^p - v_{\lambda + \ve}^p)_+ \,dV_g + \int_{B_{R+1} \setminus B_{R}} |\nabla_g \eta|_g \left(|\nabla_g u|_g^{p-1}u + |\nabla_g \log v_{\lambda+\ve}|_g^{p-1} u^{p}\right) \,dV_g \nn \\& \le \int_{B_{R+1} \setminus B_{R}} |\nabla_g \eta|_g \left(|\nabla_g u|_g^{p-1}u + |\nabla_g \log v_{\lambda+\ve}|_g^{p-1} u^{p}\right) \,dV_g 
		\end{align} 
		where, $\mathscr{B}\left(u,v_{\lambda + \ve}\right)$ is as in \eqref{R2picone}. On the RHS of \eqref{R2eqw3} we use the estimates $|\nabla_g \eta|_g \le c$ (for all $R > 0$), 
		$$\int_{B_{R+1}\setminus B_R} |\nabla_g u|_g^{p-1}u \,dV_g \le \lVert \nabla_g u \rVert_{L^p(B_{R+1}\setminus B_R)}^{p-1} \lVert  u \rVert_{L^p(B_{R+1}\setminus B_R)} \to 0 $$ 
		as $R \to \infty$ and $|\nabla_g \log v_{\lambda + \ve}|_g \le A_{\lambda + \epsilon}$ in $\Bn$.
		Using the Picone inequality \eqref{R2picone} on the LHS of \eqref{R2eqw3} and letting $R \to \infty$ on the RHS of \eqref{R2eqw3} we get 
		\begin{align}\label{R2eqw4} 
			C_p \int_{\Bn \setminus B_{R_{\lambda+ \ve}} \cap \left\{u \ge v_{\lambda + \ve}\right\}} v_{\lambda + \ve}^p\left(|\nabla_g \log u|_g + |\nabla_g \log v_{\lambda +\ve}|_g\right)^{p-2}|\nabla_g (\log u - \log v_{\lambda + \ve})|_g^2 \,dV_g \le 0.
		\end{align} 
	    Since, $\nabla_g v_{\lambda + \ve} \neq 0$ in $\Bn \setminus \{0\}$, from \eqref{R2eqw4} we must have $\log u - \log v_{\lambda + \ve} = C_0$ in $\Bn \setminus B_{R_{\lambda+ \ve}} \cap \left\{u \ge v_{\lambda + \ve}\right\}$. Since, $u \le v_{\lambda + \ve}$ on $\partial B_{R_{\lambda+ \ve}}$ we conclude $C_0 = 0$ i.e., $u = v_{\lambda + \ve}$ in $\Bn \setminus B_{R_{\lambda+ \ve}} \cap \left\{u \ge v_{\lambda + \ve}\right\}$. Thus $u \le v_{\lambda + \ve}$ in $\Bn \setminus B_{R_{\lambda + \ve}}$ as well. \end{proof}
    \,
	\\ Now we proceed with the proof of sharp decay estimate and gradient estimate for a solution of \eqref{R21eqn1}.
	\begin{theorem}\label{R2sharpdecay} 
		Let, $u$ be a solution of \eqref{R21eqn1}. Then there are  constants $C_1, C_2 > 0$ and $m \ge 2$ (chosen large enough depending on the exponent $q$) such that 
		\begin{align}\label{R2sharp1} C_1\left[\sinh \left(\frac{\dist_{\Bn}(0,x)}{2}\right)\right]^{-2\al} \le u(x) \le C_2\left[\cosh \left(\frac{\dist_{\Bn}(0,x)}{m}\right)\right]^{-m\al}, \, \forall \, x \in \Bn \setminus B_{\overline{R}}
		\end{align} 
	for some $\overline{R} > 0$. 
    \end{theorem}
	\begin{proof} We prove the upper bound. The proof of lower bound is much simpler and follows similarly. 
		\\ Recall from Lemma-\ref{R2lemmasubsol} that $\dis v(x) := C_2\left[\cosh \left(\frac{\dist_{\Bn}(0,x)}{m}\right)\right]^{-m\al}$ is a super-solution to the equation 
		\begin{align}\label{R2comp1} 
			-\Delta_p^{\Bn} v - \lambda v^{p-1} \ge c_\lambda e^{-\frac{2}{m}\dist_{\Bn}(0,x)} v^{p-1} \text{ in } \Bn\setminus B_{R_\lambda} 
		\end{align} 
	    for $C_2 > 0$ and $m \ge 2$.
		\\ Let $\eta \in C_c^\infty(\Bn)$ be a non-negative test function such that $\eta \equiv 1$ in $B_R$ and $\eta \equiv 0$ in $\Bn \setminus B_{R+1}$ where, $R > \overline{R} > R_\lambda$ ($\overline{R}$ to be determined later).
		\\ Then testing equation \eqref{R21eqn1} with $\varphi_1 := \eta u^{1-p}(u^p - v^p)_+$ and equation \eqref{R2comp1} with $\varphi_2 := \eta v^{1-p}(u^p - v^p)_+$ and subtracting we have 
		\begin{align}\label{R2comp2} 
			& \int_{\Bn \cap \left\{u \ge v\right\}} \eta \mathscr{B}(u,v) \,dV_g \nn \\ &\le \int_{\Bn \cap \left\{u \ge v\right\}} \eta \left(u^{q - p} - c_\lambda e^{-\frac{2}{m}\dist_{\Bn}(0,x)}\right)(u^p - v^p)_+ \,dV_g \nn \\ & \quad + \int_{B_{R+1} \setminus B_{R}} |\nabla_g \eta|_g \left(|\nabla_g u|_g^{p-1}u + |\nabla_g \log v|_g^{p-1} u^{p}\right) \,dV_g := I_1 + I_2, 
		\end{align}
	    where, $\mathscr{B}(u,v)$ is as in \eqref{R2picone}.
		\\ Now, using the estimate from Lemma-\ref{R2weakdecay} we may choose $\overline{R} > 0$ large enough such that $$u^{q-p} (x) \le C_{\lambda+\ve}^{q-p}e^{-(q-p)\alpha_{\lambda + \ve}\dist_{\Bn}(0,x)} \le c_\lambda e^{-\frac{2}{m}\dist_{\Bn}(0,x)}$$ for $x \in \Bn \setminus B_{\overline{R}}$ (the later is ensured by choosing $m > \frac{2}{(q-p)\alpha_{\lambda + \ve}}$) and choose $C_2 > 0$ large enough such that $u \le v$ in $B_{\overline{R}}$. Thus, $I_1 \le 0$ and estimating as in the proof of Lemma-\ref{R2weakdecay} we have $I_2 \to 0$ as $R \to \infty$. Thus from the Picone inequality \eqref{R2picone} we get \begin{align} C_p \int_{\Bn \setminus B_{\overline{R}} \, \cap \, \left\{u \ge v\right\}} v^p \left(|\nabla_g \log u|_g + |\nabla_g \log v|_g\right)^{p-2}|\nabla_g (\log u - \log v)|_g^2 \,dV_g \le 0. \end{align} Since, $\nabla_g v \neq 0$ in $\Bn \setminus \left\{0\right\}$ we get the desired upper bound \eqref{R2sharp1} as before. \end{proof}
	
	\begin{theorem}\label{R2gradestimate} Let, $u$ be a positive solution of \eqref{R21eqn1}. Then $\exists \, A > 0$ such that the gradient estimate \begin{align} A^{-1} u(x) \le |\nabla_g u|_g(x) \le A u(x) \end{align} holds in $x \in \Bn \setminus B_{R_0}$ for some $R_0$ large. In fact we have $\lim\limits_{\dist_{\Bn}(0,x) \to \infty} \frac{|\nabla_g u|_g}{u}(x) = \al$, where, $\al$ is as in \eqref{R2root}. \end{theorem}
	
	\begin{proof} 
		
		Using the sharp decay estimates Theorem-\ref{R2sharpdecay} on a positive solution of \eqref{R21eqn1} we have $$ce^{-\alpha_\lambda \dist_{\Bn}(0,x)} \le u(x) \le Ce^{-\alpha_\lambda \dist_{\Bn}(0,x)}$$ in $\Bn$ for some positive constants $c, C > 0$. 
		\\ Let, $x_0 \in \Bn \setminus B_{R_0 + 1}$ and define $v_0(x) = (u \circ \tau_{0})$ where, $\tau_{0} \in \mathcal{M}(\Bn)$ is a Hyperbolic reflection such that $\tau_{0}(0) = x_0$ and $\tau_{0}(x_0) = 0$.
		Then from equation \eqref{R21eqn1} we get 
		\begin{align} -\Delta_p^{\Bn} v_0 = (\lambda + v_0^{q-p}) v_0^{p-1} = h_0 v_0^{p-1} \end{align} where, $h_0 = \lambda + O(e^{-(q-p) \alpha_\lambda R_0})$ in $B_1$. Then using $C^{1,\sigma}$ estimate (see \cite{DiBe}, \cite{Tol}) and Harnack inequality (see \cite{T}) for this equation we have \begin{align} \sup_{B_{1/2}} |\nabla_g v_0|_g \le C_1\sup_{B_{1}} \, v_0 \le C_2 v_0(0). \end{align} \\ Reverting back to $u$ we get the estimate \begin{align} |\nabla_g u|_g(x_0) \le A u(x_0) \end{align} where the constant $A$ is independent of $x_0 \in \Bn \setminus B_{R_0 + 1}$. \\\\ Now to see the gradient estimate from below, suppose to the contrary there is a sequence of points $ \{x_k\}_{k \in \mathbb{N}} \subset \Bn$ such that $|\nabla_g u|_g(x_k) \le \epsilon_k u(x_k)$ and $R_k := \dist_{\Bn}(0,x_k) \to \infty$ where, $\epsilon_k \to 0^+$ as $k \to \infty$. Define $v_k := \dfrac{(u \circ \tau_{k})}{u(x_k)}$ where, $\tau_k \in \mathcal{M}(\Bn)$ is a Hyperbolic reflection such that $\tau_{k}(0) = x_k$ and $\tau_k(x_k) = 0$. Then $v_k(0) = 1$, $|\nabla_g v_k|_g(0) \le \epsilon_k$ and $v_k$ satisfies the equation \begin{align} \label{R2link} -\Delta_p^{\Bn} v_k = h_k v_k^{p-1} \end{align}  where, $h_k := \lambda + (u \circ \tau_{k})^{q-p} \to \lambda$ uniformly in $B_R$ as $k \to \infty$ for each $R > 0$. Furthermore, using the sharp decay estimate of $u$ we have $v_k$ satisfies the pointwise bounds 
		\begin{align}\label{R2boundk}
			C_1 E_k(x) \le v_k(x) \le C_2E_k(x) 
		\end{align} 
	    for $x \in B_R$ and the positive constants $C_1, C_2 > 0$ (independent of $R$) where, \begin{align} E_k(x) := \left(\sinh \left(\frac{1}{2}\dist_{\Bn}(0,x_k)\right)\right)^{2\alpha_\lambda} \left(\sinh \left(\frac{1}{2}\dist_{\Bn}(x,x_k)\right)\right)^{-2\alpha_\lambda} = \left(\frac{|x_k|^2(1 - |x|^2)}{|x - x_k|^2}\right)^{\alpha_\lambda} \end{align} uniformly for all $k \ge K_R$ sufficiently large. Here we used the fact that 
		\begin{align}
			(u \circ \tau_k)(x) \approx e^{-\al \dist_{\Bn}(0,\tau_k(x))} = e^{-\al \dist_{\Bn}(\tau_k(0),x)} \approx \left(\sinh \left(\frac{1}{2}\dist_{\Bn}(x_k,x)\right)\right)^{-2\al}
		\end{align}
	    for $x \in B_R$ with uniform constants (independent of $R$) for all $k \ge K_R$ sufficiently large and the identity (see \cite{stoll}, equation (2.2.4))
	    \begin{align}
	    	\sinh^2\left(\frac{1}{2}\dist_{\Bn}(x,y)\right) = \frac{|x - y|^2}{(1 - |x|^2)(1 - |y|^2)}, \, \forall \, x,y \in \Bn.
	    \end{align}
		\\\\ Then using $C^{1,\sigma}$ estimate on \eqref{R2link} in $B_R$ we get $[v_k]_{C^{1,\sigma}(B_R)} \le C_R\sup_{B_R} v_k$ where, the constant $C_R$ only depends on $B_R$ for all $k \ge K_R$ sufficiently large. Passing through a subsequence if necessary (which we still denote as $\{v_k\}_{k \in \mathbb{N}}$) we get $v_k \to v_{R,\infty}$ in $C^1(B_R)$ and $x_k \to \xi \in \partial_{\infty} \Bn (= \mathbb{S}^{n-1})$ as $k \to \infty$. Then $v_{R,\infty}$ satisfies the equation \begin{align} \label{R2lininfty} -\Delta_p^{\Bn}v_{R,\infty} = \lambda v_{R,\infty}^{p-1}. \end{align} Letting $R \to \infty$ and using a diagonalization argument we have a subsequence $v_k \to v_\infty$ locally uniformly in $C^1(\Bn)$ such that $v_\infty$ satisfies the equation \begin{align} -\Delta_p^{\Bn} v_\infty = \lambda v_\infty^{p-1} \end{align} in $\Bn$ and from \eqref{R2boundk} the pointwise bounds \begin{align} C_1 E_\xi(x) \le v_\infty(x) \le C_2E_\xi(x) \end{align} where, $E_\xi(x) := \left(\frac{1 - |x|^2}{|x - \xi|^2}\right)^{\alpha_\lambda}$ with $v_\infty(0) = \lim_{k \to \infty} v_k(0) = 1$ and $|\nabla_g v_\infty|_g(0) = \lim\limits_{k \to \infty} |\nabla_g v_k|_g(0) = 0$. \\\\ But from the classification of eigenfunctions from Theorem-\ref{R2classeigen} we know $v_\infty \equiv E_\xi$, which contradicts $|\nabla_g E_\xi|_g(0) = \alpha_\lambda > 0$. This completes the proof of the gradient estimate from below. By the same argument let $\{x_k\}_{k \in \mathbb{N}}$ be any sequence of points such that $\dist_{\Bn}(0,x_k) \to \infty$, then we can extract a subsequence $\{x_{k_j}\}$ such that $\frac{|\nabla_g u|_g}{u}(x_{k_j}) \to |\nabla_g E_{\xi}|_g(0) = \al$ for some $\xi \in \mathbb{S}^{n-1}$. Therefore, the limit of $\frac{|\nabla_g u|_g}{u}(x)$ as $\dist_{\Bn}(0,x) \to \infty$ exists and equals $\al$. \end{proof}
		\,
		\\ Keeping in continuation with the previous notation we state Theorem-\ref{R2eigenclfn} in the Poincar\'{e} ball model $\Bn$. The proof is mostly carried out in the equivalent upper half-space model of the Hyperbolic space, but we expand upon the parallel picture in the ball model in Remark-\ref{R2rmkballmodel}.
	    \begin{theorem}\label{R2classeigen} Let, $\lambda \in[0,\lambda_{\text{max}})$ and $v$ be a $C^1$ solution of the equation \begin{align}\label{R2eigeneq} -\Delta_p^{\Bn} v = \lambda v^{p-1} \end{align} in $\Bn$ such that there exists positive constant $C > 0$ and a point $\xi \in \mathbb{S}^{n-1}$ such that \begin{align} E_{\xi} \le v \le CE_\xi \end{align} where, $E_\xi(x) := \left(\frac{1 - |x|^2}{|x - \xi|^2}\right)^{\alpha_\lambda}$. Then $\exists \, c_0 \in [1,C] $ such that $v \equiv c_0E_\xi$ in $\Bn$. This in particular implies $\frac{|\nabla_g v|_g}{v} = \frac{|\nabla_g E_\xi|_g}{E_\xi} \equiv \al$ in $\Bn$.
	    \end{theorem} 
    
    \begin{proof} It is more convenient to prove the result in the upper half-space model of the Hyperbolic space. Using an orthogonal transformation in $\Bn$ about origin without loss of generality we may assume $\xi = -e_n$. Recall that the extended Mobius map $\Phi :  \overline{\Rn_+} \cup \{\infty\} \rightarrow \overline{\Bn}$ defined by 
    \begin{align} 
    	\Phi(x) := \left(\frac{2x_1}{|x+e_n|^2}, \cdots, \frac{2x_{n-1}}{|x+e_n|^2}, \frac{1 - |x|^2}{|x+e_n|^2} \right) \text{ and } \Phi(\infty) := -e_n 
    \end{align} 
    defines an isometry between the Poincar\'{e} ball model $\Bn$ and the upper half-space model $\Hn$ of the Hyperbolic space. In fact $\Phi$ is the inversion map w.r.t. the sphere centered at $-e_n$ and radius $\sqrt{2}$ given by $\Phi(x) = -e_n + \frac{2(x + e_n)}{|x + e_n|^2}$ and it satisfies $\Phi \circ \Phi = \operatorname{Id}$.
    \\\\ Let, $u := v \circ \Phi$. Then the problem can be equivalently posed in upper half-space model as 
    \begin{align}\label{R2eigen1} 
    	-\Delta_p^{\Hn} u = \lambda u^{p-1} 
    \end{align} 
    and satisfies the bounds 
    \begin{align}\label{R2eigenbound1}
    	x_n^{\alpha_\lambda} \le u(x) \le Cx_n^{\alpha_\lambda}. 
    \end{align}
    We claim that $u(x) = c_0x_n^{\al}$ for some $c_0 \in [1,C]$, from which our claim about the eigenfunction in $\Bn$ model will follow. We start by showing that $u$ has one dimensional symmetry in the $e_n$-direction. This then reduces the problem to showing uniqueness (up to constant multiples) of solutions to an ordinary differential equation.
    \\ \textbf{Step-I:} Note that due to the Hyperbolic isometry invariance and homogeneity of the equation \eqref{R2eigeneq} by applying the standard $C^{1,\sigma}$ estimate and the Harnack inequality we have the global gradient estimate $$|\nabla_g u|_g(x) \le A u(x)$$ for all $x \in \Hn$. 
    \\ For $x_0 = (x_0',0) \in \partial \Rn_+$ and $R > 0$, let us denote the Hyperbolic half-spaces $$H_R(x_0) := \left\{x \in \Rn_+: |x-x_0| < R,\, x_n > 0\right\}$$ and its boundary the orthogonal sphere (orthogonal to the boundary $\partial \mathbb{R}^n_+$) $$S_R(x_0) := \left\{x \in \Rn_+: |x-x_0| = R,\, x_n > 0\right\}.$$ We start by noting that $\lVert \nabla_g u \rVert_{L^p(H_R(x_0))}^p = \int_{H_R(x_0)} |\nabla_g u|_g^p x_n^{-n}\,dx \le (AC)^p\int_{H_R(x_0)} x_n^{p \alpha_\lambda - n} \,dx < +\infty$ since, $\alpha_\lambda > \frac{n-1}{p}$. 
    \\ Let $\tau \in \mathcal{M}(\Hn)$ denote the reflection across the orthogonal sphere $S_R(x_0)$ i.e., $$\tau(x) := x_0 + \frac{R^2(x - x_0)}{|x - x_0|^2}.$$ Then $u_\tau := (u \circ \tau)$ satisfies the equation 
    \begin{align}\label{R2eigen2}
    	-\Delta_p^{\Hn} u_\tau = \lambda u_\tau^{p-1} 
    \end{align} 
    along with the  bound 
    \begin{align}\label{R2eigenreflbound1}
    	\left(\frac{R^2 x_n}{|x - x_0|^2}\right)^{\alpha_\lambda} \le u_\tau(x) \le C\left(\frac{R^2 x_n}{|x - x_0|^2}\right)^{\alpha_\lambda}. 
    \end{align} 
    \\ Comparing \eqref{R2eigenbound1} and \eqref{R2eigenreflbound1} we have $u(x) < u_\tau(x)$ in the set $H_{R'}(x_0)$ where, $R' := RC^{-\frac{1}{2\alpha_\lambda}} < R$. Using the fact that $\tau \in \mathcal{M}(\Hn)$ is an isometry we also have $$\Vert \nabla_g u_\tau \rVert_{L^p(H_R(x_0) \setminus H_{R'}(x_0))} = \Vert \nabla_g u \rVert_{L^p(H_{R''}(x_0) \setminus H_{R}(x_0))} \le \Vert \nabla_g u \rVert_{L^p(H_{R''}(x_0))} < +\infty$$ where, $R'' := R^2/R' = RC^{\frac{1}{2\al}} > R$. 
    \\ We note that the functions $\phi_1 := u^{1-p}(u^p - u_\tau^p)_+\chi_{H_R(x_0)}$ and $\phi_2 := u_\tau^{1-p}(u^p - u_\tau^p)_+\chi_{H_R(x_0)}$ are supported in $H_R(x_0)\setminus H_{R'}(x_0)$, since $u = u_\tau$ on $\partial H_R(x_0)$ and $u < u_{\tau}$ in $H_{R'}(x_0)$. Furthermore, using the bounds on $u$ we have $$u(x) \approx u_\tau(x) \approx \left(\frac{R^2x_n}{|x - x_0|^2}\right)^{\al} \approx x_n^{\al}$$ in $H_R(x_0)\setminus H_{R'}(x_0) = \left\{x \in \Rn_+: R' \le |x - x_0| < R\right\}$. By Harnack's inequality we also have $|\nabla_g u|_g \lesssim u \lesssim x_n^{\al}$ and $|\nabla_g u_\tau|_g \lesssim u_\tau \lesssim x_n^{\al}$ in this set. Then we note that \begin{align*} |\nabla_g \phi_2|_g &= \left|(u^p - u_{\tau}^p)_+ \nabla_g (u_\tau^{1-p}) + u_\tau^{1-p} \nabla_g (u^p - u_{\tau}^p)_+ \right|_g \nn \\& \le (u^p - u_{\tau}^p)_+ \left|\nabla_g (u_\tau^{1-p})|_g + u_\tau^{1-p} |\nabla_g (u^p - u_{\tau}^p)_+ \right|_g \nn \\ &\le (p-1)u_{\tau}^{-p}(u^p - u_{\tau}^p)_+ |\nabla_g u_\tau|_g + pu_\tau^{1-p}u^{p-1}|\nabla_g u|_g + pu_\tau^{1-p}u_\tau^{p-1}|\nabla_g u_\tau|_g \nn \\ & \le \left((p-1)\left(\frac{u}{u_{\tau}}\right)^{p} + p\right) |\nabla_g u_\tau|_g + p\left(\frac{u}{u_{\tau}}\right)^{p-1}|\nabla_g u|_g \lesssim u_\tau + u \lesssim x_n^{\al} \nn \end{align*} 
    in $H_R(x_0)\setminus H_{R'}(x_0)$. Similarly, we have $|\nabla_g \phi_1|_g \lesssim x_n^{\al}$ in $H_R(x_0)\setminus H_{R'}(x_0)$. Therefore, $\phi_j \in \mathcal{D}^{1,p}(\Hn)$ for $j=1,2$ are both admissible test functions.
    \\ We test equations \eqref{R2eigen1} with $\phi_1 = u^{1-p}(u^p - u_\tau^p)_+\chi_{H_R(x_0)}$ and \eqref{R2eigen2} with $\phi_2 = u_\tau^{1-p}(u^p - u_\tau^p)_+\chi_{H_R(x_0)}$, subtracting and combining with Lemma-\ref{R2piconelemma} we get 
    \begin{align}\label{R2eigen3}
    	 C_p\int_{H_R(x_0) \cap \{u \ge u_\tau\}} u_\tau^p \left(|\nabla_g \log u|_g + |\nabla_g \log u_\tau|_g\right)^{p-2} |\nabla_g (\log u - \log u_\tau)|_g\,dV_g \le 0. 
    \end{align} 
    Then, $(\log u - \log u_\tau)$ is a constant in $H_R(x_0) \cap \{u \ge u_\tau\}$. Since, $u = u_\tau$ on $\partial H_R(x_0) \cap \{x_n > 0\}$ we conclude $u = u_\tau$ in $H_R(x_0) \cap \{u \ge u_\tau\}$. Therefore, we have $u_\tau \ge u$ in $H_R(x_0)$. \\ This in particular implies the outward normal derivative of $u$ on the orthogonal sphere $S_R(x_0)$ is non-negative i.e., $\frac{(x - x_0)}{|x - x_0|} \cdot \nabla u(x) \ge 0$ for each $x \in S_R(x_0)$. Fixing the point $x \in \Hn$ we may vary the orthogonal spheres $S_R(x_0)$ passing through $x$ instead where, $x_0 \in \partial \Rn_+$ and $R > 0$. Then let us denote $x = (x',x_n) \in \mathbb{R}^{n-1} \times \mathbb{R}_+$ and for each $\xi \in \mathbb{S}^{n-2}$ consider the inequality on the orthogonal sphere $S_{(x_n^2 + t^2)^{1/2}}((x' - t\xi,0))$ with centre at $(x' - t\xi,0) \in \partial \Rn_+$ and passing through $x$, we get $\frac{(t\xi,x_n)}{(x_n^2 + t^2)^{1/2}}\cdot \nabla u (x) \ge 0$. Letting $t \to +\infty$ we get $(\xi,0) \cdot \nabla u(x) \ge 0$ for all $\xi \in \mathbb{S}^{n-2}$ i.e., $\nabla u = (0',u_{x_n})$ or the horizontal component of gradient of $u$ is zero. Hence, $u(x) := w(x_n)$ is a function of a single variable and looking at $t = 0$ we get $u_{x_n} \ge 0$.
    \\\\ \textbf{Step-II:} As a result the equation \eqref{R2eigen1} reduces to the following ordinary differential equation 
    \begin{align}\label{R2eigen4} 
    	-t^n(t^{p-n} (w')^{p-1})' = \lambda w^{p-1} 
    \end{align} 
    on $\R_+$ where, $w' \ge 0$ and satisfies the global bounds 
    \begin{align} \label{R2eigen4a} 
    	t^{\alpha_\lambda} \le w(t) \le Ct^{\alpha_\lambda}. 
    \end{align} 
    Note that when $\lambda = 0$ we can directly integrate both sides of the ODE and conclude $w(t) = c_0t^{\frac{n-1}{p-1}}$ as $\alpha_0 = \frac{n-1}{p-1}$. Therefore, it remains to prove the result for $\lambda > 0$ i.e., when $\frac{n-1}{p} < \alpha_\lambda < \frac{n-1}{p-1}$. 
    \\ From the estimate $|\nabla_g u|_g \le Au$ we get $tw'(t) \le Aw(t) \le C't^{\alpha_\lambda}$. We have $t^{p-n}(w')^{p-1} \le C''t^{p-n+(p-1)(\alpha_\lambda - 1)} = C''t^{(p-1)\alpha_\lambda - (n-1)} \to 0$ as $t \to \infty$ for $\lambda > 0$. Integrating equation \eqref{R2eigen4} from $t$ to $\infty$ we get 
    \begin{align} 
    	t^{p-n} (w'(t))^{p-1} = -\int_t^\infty (s^{p-n} (w')^{p-1})'\,ds &= \lambda \int_t^\infty w^{p-1}s^{-n}\,ds \nn \\ &\ge \lambda \int_t^\infty s^{(p-1)\alpha_\lambda - n}\,ds \nn \\& = \alpha_\lambda^{p-1} t^{(p-1)\alpha_\lambda - (n-1)} 
    \end{align} i.e., $w'(t) \ge \alpha_\lambda t^{\alpha_\lambda - 1}$. Therefore using the bounds of $w$ we conclude that 
    \begin{align}\label{R2eigen5}
    	0 < \alpha := \inf_{t > 0} \frac{tw'(t)}{w(t)} \le \frac{tw'(t)}{w(t)} \le \sup_{t > 0} \frac{tw'(t)}{w(t)} := \beta < \infty. 
    \end{align} 
    \\ We note from \eqref{R2eqn44} that the linearized equation corresponding to \eqref{R2eigen1} is satisfied by $(x.\nabla u)$ (since, $u(x) = w(x_n)$ has no critical points in $\Hn$) and is given by 
    \begin{align}\label{R2eigen6} & \int_{\Hn} x_n^{p-n}\left(|\nabla u|^{p-2}(\nabla (x.\nabla u), \nabla \varphi) + (p-2)|\nabla u|^{p-4}(\nabla (x.\nabla u), \nabla u)(\nabla u, \nabla \varphi)\right)\,dx \nn \\& = (p-1)\lambda \int_{\Hn} u^{p-2}(x. \nabla u) \varphi x_n^{-n}\,dx, \, \forall \, \varphi \in C_c^\infty(\Hn). 
    \end{align} 
    Since, $u(x) = w(x_n)$ the above PDE reduces to the ordinary differential equation \begin{align}\label{R2eigen7} 
    	-t^n(t^{p-n}(w')^{p-2}(tw')')' = \lambda w^{p-2}(tw'). 
    \end{align}
    \\ Let, $\tilde{w}_1 := tw' - \alpha w$ and $\tilde{w}_2 := \beta w - tw'$, then from \eqref{R2eigen5} these are non-negative functions. Combining \eqref{R2eigen4} and \eqref{R2eigen7} we see that $\tilde{w}_j$ for $j = 1,2$ satisfies 
    \begin{align} \label{R2eigen8} -t^n(t^{p-n}(w')^{p-2}(\tilde{w}_j)')' = \lambda w^{p-2}(\tilde{w}_j). 
    \end{align}
    \\ From the strong minimum-principle (see Lemma-\ref{R2harnack1}) we conclude either $\tilde{w}_j \equiv 0$ in $\R_+$ or has strict inequality $\tilde{w}_j > 0$ in $\R_+$. In the first case clearly we must have $w(t) = c_0t^{\alpha_\lambda}$ keeping in mind \eqref{R2eigen4a}. In the second case neither the infimum or supremum in \eqref{R2eigen5} are attained in interior of $\R_+$. Therefore, without loss of generality let us assume $$\alpha := \inf_{t > 0} \frac{tw'(t)}{w(t)} = \liminf\limits_{t \to 0^+} \frac{tw'(t)}{w(t)}.$$ Let $\{R_k\}_{k \in \mathbb{N}}$ be a sequence in $\R_+$ with $R_k \to 0^+$ such that $\frac{R_kw'(R_k)}{w(R_k)} \to \alpha$. Then consider the sequence of functions $w_k(t) := R_k^{-\al} w(R_k t)$ which satisfy the equation \eqref{R2eigen4} and the global bounds \eqref{R2eigen4a}. Note that $\inf_{t > 0} \frac{tw_k'(t)}{w_k(t)} = \inf_{t > 0} \frac{tw'(t)}{w(t)} = \alpha$ for all $k \in \mathbb{N}$. Applying the standard $C^{1,\alpha}$-estimates \cite{DiBe}, \cite{Tol} we may extract a subsequence (which we still index with $k$) $w_k \to w_\infty$ in $C^{1}(\R_+)$ locally uniformly such that $w_\infty$ satisfies the equation \eqref{R2eigen4}, the global bounds \eqref{R2eigen4a} and $\lim_{k \to \infty} \frac{w_k'(1)}{w_k(1)} = \frac{w_\infty'(1)}{w_\infty(1)} = \alpha$. Also from the $C^1(\R_+)$ local uniform convergence we have $\inf_{t > 0} \frac{t w_\infty'(t)}{w_\infty(t)} \ge \alpha$. Therefore, $\tilde{w}_\infty(t) := tw'_\infty(t) - \alpha w_\infty(t)$ is a non-negative solution of \eqref{R2eigen8} such that $\tilde{w}_\infty(1) = 0$. From the strong minimum-principle (see Lemma-\ref{R2harnack1}) we conclude $\tilde{w}_\infty \equiv 0$ in $\R_+$ i.e., $w_\infty(t) = ct^\alpha$ in $\R_+$ for some $c >0$ (recall $\alpha> 0$ from \eqref{R2eigen5} so that $w_\infty$ has no critical points in $\R_+$). But comparing with the global bounds \eqref{R2eigen4a} of $w_\infty$ we must have $\alpha = \al$. By a similar consideration for the supremum $\beta$ in \eqref{R2eigen5} we also have $\beta = \al$. Therefore, $\alpha = \beta = \alpha_\lambda$, which now implies $w(t) = c_0t^{\alpha_{\lambda}}$ for some $c_0 > 0$. This completes the proof of classification of solutions to the ordinary differential equation. \end{proof}
    
    \begin{remark}\label{R2rmkballmodel} It is also interesting to see the proof of symmetry in Theorem-\ref{R2classeigen} directly in the ball model $\Bn$. Let $H$ denote an open half-space in $\Bn$, i.e. the boundary of $H$ is an orthogonal sphere $\Sigma := \partial H$ that splits $\Bn$ into two isometric components $H$ and $\Bn \setminus \overline H$. Let us denote the Hyperbolic reflection w.r.t. $\Sigma$ as $\sigma \in \mathcal{M}(\Bn)$, then $\Bn \setminus \overline H = \sigma(H)$. We define the ideal-boundary of $\Sigma$ as $\partial_{\infty} \Sigma := \partial_{\infty} \overline{H} \setminus \partial_{\infty} H$, that is the part of the ideal boundary $\partial_{\infty} \Bn$ formed by geodesic rays restricted to $\Sigma$. 
    \\ Now, given $\xi \in \partial_{\infty} \Bn$ as in statement of Theorem-\ref{R2classeigen}, if we denote $u_\sigma = (u \circ \sigma)$, then the weak-comparison argument in the first half of the proof shows that $u_\sigma \ge u$ in $\overline H$, for any half-space $H$ such that $\xi \not\in \partial_{\infty} \overline{H}$. Let, $\Sigma_1$ be an orthogonal sphere with $\xi \in \partial_{\infty} \Sigma_1$, that splits $\Bn$ into the two half-spaces $H_1$ and $H_2 = \Bn \setminus \overline{H_1}$. We denote the reflection w.r.t. $\Sigma_1$ as $\sigma_1 \in \mathcal{M}(\Bn)$. Note that $u_\sigma \ge u$ in all half-spaces $\overline{H} \subset H_1$ such that $\xi \not\in \partial_\infty \overline{H}$. So by continuity of $u$ we have $u_{\sigma_1} \ge u$ in $H_1$. Similarly, we also have $u_{\sigma_1} \ge u$ in $H_2$. Combining this information we get $u = u_{\sigma_1}$ in $\Bn$ for all $\Sigma_1$ such that $\xi \in \partial_\infty \Sigma_1$. Now the set of reflections in orthogonal spheres $\Sigma_1$ with $\xi \in \partial_\infty \Sigma_1$ generate the parabolic subgroup of isometries in $\mathcal{M}(\Bn)$ that fix $\xi \in \partial_\infty \Bn$ (for more details we refer to Section-4.7: Classification of Transformations in \cite{Ratcli}). Hence the level sets of $u$ are horospheres with geometric center $\xi \in \partial_\infty \Bn$, i.e. $u$ has horospherical symmetry.
    \\ Under an isometric map that sends $\Bn$ to the upper half-space model $\Hn$ of the Hyperbolic space such that $\xi \in \partial_{\infty} \Bn$ is mapped to $+\infty \in \partial_{\infty} \Hn$, we note that horospheres with geometric center $\xi$ are mapped to Euclidean hyperplanes $\{x_n = c\}$ in $\Rn_+ ( = \Hn)$. This is indeed why we saw the 1-dimensional symmetry of $u$ in the proof of Theorem-\ref{R2classeigen} in the upper half-space model of the Hyperbolic space.
    	
    \end{remark}

    \begin{remark}\label{R2rmkminimal}
    	For $\lambda \in [0,\lambda_{\text{max}})$, each $\xi \in \partial_\infty \Bn (=\mathbb{S}^{n-1})$ has at least two distinct classes (upto constant multiple) of corresponding positive eigenfunctions solving \eqref{R2eigeneq}, given by \begin{align}
    		E(x,\xi) = \left(\frac{1 - |x|^2}{|x - \xi|^2}\right)^{\alpha_\lambda} \text{ and } F(x,\xi) = \left(\frac{1 - |x|^2}{|x - \xi|^2}\right)^{\beta_\lambda}
    	\end{align}
    	where, $\al$ and $\beta_\lambda$ are as in \eqref{R2root}. The important distinction between them is that $\lVert \nabla_g E(.,\xi) \rVert_{L^p(\overline{H})} < +\infty$ for every Hyperbolic closed half-space $\overline{H}$ of $\Bn$ such that $\xi \not\in \partial_\infty \overline{H}$ i.e., $E(.,\xi)$ has finite energy in each closed half-space $\overline{H}$ which does not contain the `singularity' $\xi$ on its ideal boundary. On the other hand $\lVert \nabla_g F(.,\xi) \rVert_{L^p(\overline{H})} = +\infty$ for all $\overline{H}$. The finiteness of energy of $E(.,\xi)$ in each half-space $\overline{H}$ away from the `singularity' at $\xi$ plays a crucial role in the moving plane argument in proof of Theorem-\ref{R2classeigen}.
    	\\\\ We only state an extract of the representation theorem in the linear case $p = 2$ due to Helgason \cite{Helgason} and Minemura \cite{Minemura} relevant to our particular case. We denote the normalized measure on $\mathbb{S}^{n-1}$ with $d\sigma$, the space of all analytic functions on $\mathbb{S}^{n-1}$ by $C^\omega(\mathbb{S}^{n-1})$ and its dual the space of analytic functionals (hyperfunctions) by $C^{-\omega}(\mathbb{S}^{n-1})$. Let us denote by $\mathcal{L}_\lambda := -\Delta^{\Bn} - \lambda$.
        \begin{theorem}\label{R2eigenclass1}
        	The Poisson-Helgason transform $P_\lambda: C^{-\omega}(\mathbb{S}^{n-1}) \to C^\omega(\Bn)$ defined by 
        	\begin{align}\label{R2eigenclass2}
        		P_\lambda(T)(x) := \int_{\mathbb{S}^{n-1}} E(x,\xi)T(\xi)\,d\sigma(\xi), \text{ for } T \in C^{-\omega}(\mathbb{S}^{n-1})
        	\end{align}
        	is an isomorphism between $C^{-\omega}(\mathbb{S}^{n-1})$ and $\operatorname{Ker}\left(\mathcal{L}_\lambda\right)$, with inverse $B_\lambda$ given by a generalized boundary value. For $v \in \operatorname{Ker}\left(\mathcal{L}_\lambda\right)$ the hyperfunction $B_\lambda(v)$ is given by 
        	\begin{align}\label{R2eigenclass3}
        		\left<B_\lambda(v), \varphi \right> := c_{n,\lambda} \lim\limits_{r \to 1^{-}} (1 - r^2)^{\alpha_\lambda - (n-1)} \int_{\mathbb{S}^{n-1}} v(r\xi) \varphi(\xi)\,d\sigma(\xi), \text{ for } \varphi \in C^\omega(\mathbb{S}^{n-1})
        	\end{align}
        	where, $c_{n,\lambda} = \frac{\pi^{\frac{1}{2}} \Gamma\left(\alpha_\lambda\right)}{\Gamma\left(\alpha_\lambda - \frac{n-1}{2}\right)}$.
        \end{theorem}
        \, \\ For the sake of completeness we include a short proof of the minimality of $E(.,\xi)$ relative to $\xi \in \mathbb{S}^{n-1}$ in the linear case $p = 2$. Let us consider the following weighted analogue of a Hardy Space 
        \begin{align}
        	\mathrm{H}^1_{\lambda}(\Bn) := \left\{f \in L^1_{\text{loc}}(\Bn): \sup_{r \in [0,1)} (1 - r^2)^{\alpha_\lambda - (n-1)} \int_{\mathbb{S}^{n-1}} |f(r\xi)|\,d\sigma(\xi) < +\infty \right\}.
        \end{align}
        We note that if $v \in \mathrm{H}^1_{\lambda}(\Bn) \cap \operatorname{Ker}\left(\mathcal{L}_\lambda\right)$, then the set of Borel measures $$d\mu_{r}(\xi) := c_{n,\lambda} (1 - r^2)^{\alpha_\lambda - (n-1)} v(r\xi)\,d\sigma(\xi) \in M(\mathbb{S}^{n-1})$$ for $r \in [0,1)$ has uniformly bounded total variation norm. Extracting a $w^\ast$-convergent subsequence $\mu_{r_j} \overset{w^\ast}{\rightharpoonup} \mu$ as $r_j \to 1^{-}$ we note that in \eqref{R2eigenclass3} we have 
        \begin{align}\label{R2trace1}
        	\left<B_\lambda(v), \varphi \right> = \int_{\mathbb{S}^{n-1}} \varphi \,d\mu, \, \forall \, \varphi \in C^\omega(\mathbb{S}^{n-1}).
        \end{align}
        Thus \eqref{R2trace1} extends by density to all $\varphi \in C(\mathbb{S}^{n-1})$ and hence $B_\lambda(v) = \mu \in M(\mathbb{S}^{n-1})$ is given by a Borel measure. This establishes an isomorphism of $\mathrm{H}^1_{\lambda}(\Bn) \cap \operatorname{Ker}\left(\mathcal{L}_\lambda\right)$ and $M(\mathbb{S}^{n-1})$. In particular we note that positive functions in $\mathrm{H}^1_{\lambda}(\Bn) \cap \operatorname{Ker}\left(\mathcal{L}_\lambda\right)$ corresponds to non-negative measures in $M(\mathbb{S}^{n-1})$. However we note that if $v \in \operatorname{Ker}\left(\mathcal{L}_\lambda\right)$ is a positive eigenfunction then by the mean-value equality in Hyperbolic space (see \cite{stoll}, Theorem-5.5.5) we have 
        \begin{align}
        	\int_{\mathbb{S}^{n-1}} v(r\xi) \,d\sigma(\xi) = g_{\alpha_\lambda}(r\eta)v(0)
        \end{align}
        where, $r \in [0,1), \eta \in \mathbb{S}^{n-1}$ and $g_{\alpha_\lambda} \in \operatorname{Ker}\left(\mathcal{L}_\lambda\right)$ is the unique bounded radial eigenfunction with $g_{\al}(0) = 1$ given by 
        \begin{align}
        	g_{\alpha_\lambda}(x) := \int_{\mathbb{S}^{n-1}} E(x,\eta') \,d\sigma(\eta')
        \end{align}
        with asymptotic boundary behavior (see \cite{stoll}, Theorem-5.5.7 and Corollary-5.5.8) given by 
        \begin{align}
        	g_{\alpha_\lambda}(r\eta) \approx (1 - r^2)^{n-1 - \alpha_\lambda}.
        \end{align}
        This in particular implies that $\lVert v \rVert_{\mathrm{H}^1_{\lambda}(\Bn)} = \sup_{r \in [0,1)} (1 - r^2)^{\alpha_\lambda - (n-1)} \int_{\mathbb{S}^{n-1}} v(r\xi) \,d\sigma(\xi) \le C'v(0)$, i.e.,  positive eigenfunctions are in $\mathrm{H}^1_{\lambda}(\Bn)$ and hence have non-negative measure for a generalized boundary value. As a consequence if $v \in \operatorname{Ker}\left(\mathcal{L}_\lambda\right)$ is a positive eigenfunction satisfying the  estimate from above
        \begin{align}
        	0 \le v(x) \le CE(x,\xi)
        \end{align}
        for some $\xi \in \mathbb{S}^{n-1}$, then $$\left<B_\lambda(v), \varphi \right> = c_{n,\lambda} \lim\limits_{r \to 1^{-}} (1 - r^2)^{\alpha_\lambda - (n-1)} \int_{\mathbb{S}^{n-1}} v(r\eta) \varphi(\eta)\,d\sigma(\eta) = 0$$ for all $\varphi \in C(\mathbb{S}^{n-1})$ with $\operatorname{supp}(\varphi)\cap \{\xi\} = \emptyset$. Thus $B_\lambda(v)$ is a non-negative measure supported at $\xi \in \mathbb{S}^{n-1}$ and hence must be a dirac mass, $B_\lambda(v) = c_0\delta_{\xi}$ for some $c_0 \ge 0$. Therefore, $v = P_\lambda(c_0\delta_{\xi}) = c_0 E(.,\xi).$
    \end{remark}
	 
	\section{Proof of symmetry} \label{R2symmetry}
	
	We will use the Alexandrov Moving Plane method to prove symmetry of positive energy solutions to \eqref{R21eqn1} in the Poincar\'{e} ball model $(\Bn, g)$ of the Hyperbolic space. Recall that the Hyperbolic reflections are isometries of the Hyperbolic space $\Bn$ given by spherical inversions (or reflections) with respect to generalized Euclidean spheres that are orthogonal to $\mathbb{S}^{n-1}$. In particular these are either orthogonal Euclidean spheres $S(a,r) := \{x: |x - a| = r\}$ where $|a| > 1$ and $r^2 = |a|^2 - 1$ or are Euclidean hyperplanes passing through the origin of $\Bn$. 
	\\ Given a direction $\xi \in \mathbb{S}^{n-1}$ we perform the moving plane method in direction $\xi$ by reflecting along the family of orthogonal spheres parametrized by $S\left(\mu^{-1}\xi, \sqrt{|\mu|^{-2} - 1}\right)$ for $\mu \in (-1,1)\setminus \{0\}$. The limiting case $\mu = 0$ corresponds to reflecting along the Euclidean hyperplane passing through $0 \in \Bn$ and orthogonal to $\xi$. When $\mu \neq 0$ we denote the reflection across $S\left(\mu^{-1}\xi, \sqrt{|\mu|^{-2} - 1}\right)$ by $\sigma_{\mu,\xi} \in \mathcal{M}(\Bn)$, 
	\begin{align}\label{R2reflection1} 
		\sigma_{\mu,\xi}(x) := \mu^{-1}\xi + \frac{\left(|\mu|^{-2} - 1\right)\left(x - \mu^{-1}\xi\right)}{|x - \mu^{-1}\xi|^2} 
	\end{align} 
    and when $\mu = 0$ we denote 
    \begin{align}\label{R2reflection2} 
    	\sigma_{0,\xi}(x) := x - 2(x ,\xi) \xi. 
    \end{align}
	\\ Thus $S\left(\mu^{-1}\xi, \sqrt{|\mu|^{-2} - 1}\right)$ separates $\Bn$ into two isometric connected components (half-spaces). We denote by $\Sigma_\mu(\xi)$ the half-space in $\Bn$ containing $\xi$ in its boundary shared with $\partial_\infty \Bn$ ($ = \mathbb{S}^{n-1}$). We also denote the part of the sphere $S\left(\mu^{-1}\xi, \sqrt{|\mu|^{-2} - 1}\right)$ in $\Bn$ as $T_\mu(\xi) := S\left(\mu^{-1}\xi, \sqrt{|\mu|^{-2} - 1}\right) \cap \Bn$. In this notation we have $\sigma_{\mu,\xi}\left(\Sigma_\mu(\xi)\right) = \Bn \setminus \overline{\Sigma_\mu(\xi)} = \Sigma_{-\mu}(-\xi)$. 
	\\\\ Before proceeding with the proof of symmetry we need the following simple lemma comparing the image and object distances of a reflection from an observer in Hyperbolic space. The proof is similar to that of the same fact in the Euclidean space.
	\begin{lemma}\label{hrefdist}
		Let $T$ be a generalized orthogonal sphere in $\Bn$ which splits the space into two isometric components (half-spaces) $\Sigma$ and $\Bn \setminus \overline{\Sigma}$. Let $E$ (observer) be a fixed point in $\overline \Sigma$. Let $P$ (object) be a point in $\Sigma$ and we denote its reflection in $T$ as $\sigma_T(P) = P'$ (image). Then $\dist_{\Bn}(P,E) \le \dist_{\Bn}(P',E)$ with equality if and only if $E$ lies on $T$.
	\end{lemma}
	\begin{proof}
		The proof can be seen as a corollary of the Pythagoras' theorem (or more generally the law of cosines in Hyperbolic space, see Theorem-3.5.3 in \cite{Ratcli}). Let $\ell$ denote the Hyperbolic line passing through $P$ and $P'$ which intersects $T$ at $O$. Then clearly $\dist_{\Bn}(P,O) = \dist_{\Bn}(P',O)$ i.e., $O$ is the mid-point of $P$ and $P'$ on the line $\ell$. It is easier to visualize it after applying a Hyperbolic isometry to bring $O$ to the origin of $\Bn$. Then in this new frame, $T$ is a Euclidean hyperplane passing through $O$ splitting $\Bn$ into two half-spaces $\Sigma$ and $\Bn\setminus \overline{\Sigma}$. $\ell$ is an Euclidean line passing through $O$ and orthogonal to $T$. Let $F$ be the orthogonal projection of $E$ on the line $\ell$ (i.e., $F$ is intersection of $\ell$ with the unique Hyperbolic line $\ell_1$ through $E$ which is orthogonal to $\ell$). Then applying Hyperbolic Pythagoras' theorem to the triangles $\Delta PFE$ and $\Delta P'FE$ respectively we get
		\begin{align*}
			\cosh(\dist_{\Bn}(P,E)) &= \cosh(\dist_{\Bn}(P,F))\cosh(\dist_{\Bn}(E,F)) \\ 
			\cosh(\dist_{\Bn}(P',E)) &= \cosh(\dist_{\Bn}(P',F))\cosh(\dist_{\Bn}(E,F))
		\end{align*}
		respectively. Therefore the inequality $\dist_{\Bn}(P,E) \le \dist_{\Bn}(P',E)$ will follow if we see that $\dist_{\Bn}(P,F) \le \dist_{\Bn}(P',F)$. Note that the orthogonal projection $F$ of the point $E$ on the line $\ell$ must lie on the same half-space $\overline{\Sigma}$ as $E$ (the unique Hyperbolic line $\ell_1$ passing through $E \in \overline{\Sigma}$ which is orthogonal to $\ell$ must be entirely contained in $\overline{\Sigma}$). Since, $\dist_{\Bn}(P',F) = \dist_{\Bn}(P',O) + \dist_{\Bn}(O,F)$ (as $O = T \cap \ell$ lies in between the points $P' \in \Bn\setminus \overline{\Sigma}$ and $F \in \overline{\Sigma}$ on the line $\ell$) and $\dist_{\Bn}(P',O) = \dist_{\Bn}(P,O)$, using triangle inequality we get $$\dist_{\Bn}(P',F) = \dist_{\Bn}(P,O) + \dist_{\Bn}(O,F) \ge \dist_{\Bn}(P,F)$$ and consequently $\dist_{\Bn}(P,E) \le \dist_{\Bn}(P',E)$. Equality holds if and only if $F = O$ i.e., $E$ lies on $T$. This completes the proof of the claim. 
	\end{proof}
	\, \\ We now proceed with the proof of Theorem-\ref{R2hypsymmetry}.
	\begin{proof}[Proof of Theorem-\ref{R2hypsymmetry}:] Without loss of generality we fix the direction $\xi = e_1$ and denote $\Sigma_\mu(e_1)$ simply as $\Sigma_\mu$ and $T_\mu(e_1)$ as $T_\mu$. For $x \in \Bn$ denote $x_\mu := \sigma_{\mu,e_1}(x)$ and $u_\mu(x) := u\circ \sigma_{\mu,e_1}(x) = u(x_\mu)$, where $u \in \mathcal{D}^{1,p}(\Bn)$ is a positive solution to \eqref{R21eqn1} i.e., 
	\begin{align}\label{R2eqnu} 
		-\Delta_p^{\Bn} u - \lambda u^{p-1} = u^{q-1}. 
	\end{align} 
    Since $\sigma_{\mu,e_1} \in \mathcal{M}(\Bn)$ we have $u_\mu$ is also a positive solution of \eqref{R21eqn1} i.e., 
    \begin{align}\label{R2eqnumu} 
    	-\Delta_p^{\Bn} u_\mu - \lambda u_\mu^{p-1} = u_\mu^{q-1} 
    \end{align} 
    and $u_\mu \in \mathcal{D}^{1,p}(\Bn)$. 
	\\\\ \textbf{Step-I:} To start the Moving plane method we will show that $u \le u_\mu$ in $\Sigma_\mu$ for all $\mu$ sufficiently close to $1$.
	\\ Let $\varphi \in C_c^\infty(\Bn)$ be a non-negative test function such that $0 \le \varphi \le 1$ with $\varphi \equiv 1$ in $B_R$ and $\text{supp}(\varphi) \subset B_{R+1}$, $|\nabla_g \varphi|_g \le c$ for all $R$ sufficiently large.
	\\ Consider the test functions $\psi_1 := \varphi^2 u^{1-p}(u^p - u_\mu^p)_{+} \, \chi_{\Sigma_\mu}$ and $\psi_2 := \varphi^2 u_\mu^{1-p}(u^p - u_\mu^p)_{+} \, \chi_{\Sigma_\mu}$. Note that the functions are admissible as test functions since $u = u_\mu$ on $T_\mu$, therefore testing \eqref{R2eqnu} with $\psi_1$ and \eqref{R2eqnumu} with $\psi_2$ respectively and subtracting we have 
	\begin{align}\label{R23eqn1} 
		& \int_{\Sigma_\mu \cap \{u \ge u_\mu\}} |\nabla_g u|_g^{p-2}g\left(\nabla_g u, \nabla_g \psi_1\right) - |\nabla_g u_\mu|_g^{p-2}g\left(\nabla_g u_\mu, \nabla_g \psi_2\right) \,dV_g \nn \\ &= \int_{\Sigma_\mu \cap \{u \ge u_\mu\}} \varphi^2 (u^{q-p} - u_\mu^{q-p})(u^p - u_\mu^p) \,dV_g. 
	\end{align} 
	\\ In the set $\Sigma_\mu \cap \{u \ge u_\mu\}$ using the sharp decay estimate on $u$ we also have the inequality 
	\begin{align}\label{R23eqn2}
		1 \le \frac{u}{u_\mu} \le C'e^{-\alpha_{\lambda}(\dist_{\Bn}(0,x) - \dist_{\Bn}(0,x_\mu))} \le C'
	\end{align} 
    where the final inequality is a consequence of Lemma-\ref{hrefdist}, i.e., $\dist_{\Bn}(0,x_\mu) \le \dist_{\Bn}(0,x)$ for all $x \in \Sigma_\mu$ and $\mu \ge 0$ and hence is true in particular for all $\mu$ close to $1$ (here we have applied Lemma-\ref{hrefdist} with $E = 0$, $P = x_\mu$ and $P' = x$).
	\\ Note that in $\Sigma_\mu \cap \{u \ge u_\mu\}$ using Lagrange's Mean Value theorem we have 
	\begin{align*} 
		(u^{q-p} - u_\mu^{q-p})(u^p - u_\mu^p) &\le u^{-p}\left(\frac{u^q - u_\mu^q}{u - u_\mu}\right)\left(\frac{u^p - u_\mu^p}{u - u_\mu}\right)(u - u_\mu)^2 \\ & \le qp \, u^{q-2}(u - u_\mu)^2.
	\end{align*} 
    Furthermore using the inequality $\frac{b-a}{b} \le \log b - \log a \le \frac{b-a}{a}$ for $b \ge a > 0$ we may write 
    \begin{align*} 
    	(u - u_\mu)^2 \le u^2(\log u - \log u_\mu)^2. 
    \end{align*}
	Combining these two inequalities we have 
	\begin{align*} 
		(u^{q-p} - u_\mu^{q-p})(u^p - u_\mu^p) \le qp \, u^q (\log u - \log u_\mu)^2 
	\end{align*} 
    which we use on RHS of \eqref{R23eqn1} to write 
	\begin{align} \label{R23eqn3} 
		& \int_{\Sigma_\mu \cap \{u \ge u_\mu\}} |\nabla_g u|_g^{p-2}g\left(\nabla_g u, \nabla_g \psi_1\right) - |\nabla_g u_\mu|_g^{p-2}g\left(\nabla_g u_\mu, \nabla_g \psi_2\right) \,dV_g \nn \\ &\le qp \int_{\Sigma_\mu \cap \{u \ge u_\mu\}} \varphi^2 u^q (\log u - \log u_\mu)^2 \,dV_g. 
	\end{align}
	On the other hand from the LHS of \eqref{R23eqn3} we have 
	\begin{align}\label{R23eqn4}
		&\int_{\Sigma_\mu \cap \{u \ge u_\mu\}} \varphi^2 \left(|\nabla_g u|_g^{p-2}g\left(\nabla_g u, \nabla_g \left(u - \frac{u_\mu^p}{u^p}u\right)\right) + |\nabla_g u_\mu|_g^{p-2}g\left(\nabla_g u_\mu, \nabla_g  \left(u_\mu - \frac{u^p}{u_\mu^p}u_\mu\right)\right)\right) \,dV_g \nn \\ &\le 2 \underbrace{\int_{\Sigma_\mu \cap \{u \ge u_\mu\}} \varphi |\nabla_g u|_g^{p-1}|\nabla_g \varphi|_gu^{1-p}(u^p - u_\mu^p) \,dV_g}_{ = I_1} \nn \\ & \quad + 2 \underbrace{\int_{\Sigma_\mu \cap \{u \ge u_\mu\}} \varphi |\nabla_g u_\mu|_g^{p-1}|\nabla_g \varphi|_gu_\mu^{1-p}(u^p - u_\mu^p) \,dV_g }_{ = I_2} \nn \\& \quad + qp \int_{\Sigma_\mu \cap \{u \ge u_\mu\}} \varphi^2 u^q (\log u - \log u_\mu)^2 \,dV_g
	\end{align}
    In order to estimate $I_1$ and $I_2$ we may use the global gradient estimate from above in Theorem-\ref{R2gradestimate} for $u$ to write $|\nabla_g u| \le Au$, $|\nabla_g u_\mu| \le Au_\mu$ and the fact that $|\nabla_g \varphi|_g \le c$ with $\text{supp}(|\nabla_g \varphi|_g) \subset B_{R+1}\setminus B_R$ to write 
    \begin{align}\label{R23eqn5}
    	I_1 + I_2 \le 4c A^{p-1} \int_{B_{R+1}\setminus B_R} u^p \,dV_g \le Ce^{-(p\alpha_\lambda - (n-1))R} \to 0^+
    \end{align}
    as $R \to \infty$ since, $p\alpha_\lambda > n-1$. Then using the inequality from Lemma-\ref{R2piconelemma} to LHS of \eqref{R23eqn4} we have
    \begin{align}\label{R23eqn6}
    	& C_p \int_{\Sigma_\mu \cap \{u \ge u_\mu\}} \varphi^2 \min\left\{u^p, u_\mu^p\right\}\left(|\nabla_g \log u|_g + |\nabla_g \log u_\mu|_g\right)^{p-2}|\nabla_g (\log u - \log u_\mu)|_g^2 \,dV_g \nn \\ & \le I_1 + I_2 + qp \int_{\Sigma_\mu \cap \{u \ge u_\mu\}} \varphi^2 u^q (\log u - \log u_\mu)^2 \,dV_g.
    \end{align}
    Now in case $p > 2$ we may use the gradient estimate from below on $u$ along with the global gradient estimate from above for both $u$ and $u_\mu$ from Theorem-\ref{R2gradestimate} to write 
    \begin{align}\label{R2g1} 
    	(2A)^{p-2} \ge \left(|\nabla_g \log u|_g + |\nabla_g \log u_\mu|_g\right)^{p-2} \ge |\nabla_g \log u|_g^{p-2} \ge A^{2-p} 
    \end{align} 
    in case $\dist_{\Bn}(0,\Sigma_\mu) > R_0$ which is true for all $\mu$ sufficiently close to $1$. Else if $p < 2$ then we have 
    \begin{align}\label{R2g2}
    	A^{2 - p} \ge |\nabla_g \log u|_g^{p-2} \ge \left(|\nabla_g \log u|_g + |\nabla_g \log u_\mu|_g\right)^{p-2} \ge (2A)^{p-2}.\end{align} 
    Combining these along with inequality \eqref{R23eqn2}, i.e., $\frac{u}{C'} \le u_\mu = \min\{u,u_\mu\} \le u$ in $\Sigma_{\mu} \cap \{u \ge u_\mu\}$, on the LHS of \eqref{R23eqn6} we have
    \begin{align}\label{R23eqn7a} 
    	C_{p,A}' \int_{\Sigma_\mu \cap \{u \ge u_\mu\}} \varphi^2 u^p|\nabla_g (\log u - \log u_\mu)|_g^2 \,dV_g \le I_1 + I_2 + qp \int_{\Sigma_\mu \cap \{u \ge u_\mu\}} \varphi^2 u^q (\log u - \log u_\mu)^2 \,dV_g.
    \end{align}
    Therefore using Cauchy-Schwartz inequality we have
    \begin{align}\label{R23eqn7}
    	& C_{p,A}' \int_{\Sigma_\mu \cap \{u \ge u_\mu\}} u^p|\nabla_g (\varphi (\log u - \log u_\mu))|_g^2 \,dV_g \nn \\& \le 2I_1 + 2I_2 + 2qp \int_{\Sigma_\mu \cap \{u \ge u_\mu\}} u^q \varphi^2(\log u - \log u_\mu)^2 \,dV_g \nn \\ & \quad + 2C_{p,A}' \underbrace{\int_{\Sigma_\mu \cap \{u \ge u_\mu\}} |\nabla_g \varphi|_g^2 u^p (\log u - \log u_\mu)^2 \,dV_g}_{ = I_3}.
    \end{align}
    Recall also from \eqref{R23eqn2} we have the inequality $0 \le \log u - \log u_\mu \le \frac{u - u_\mu}{u_\mu} \le (C' - 1)$ in $\Sigma_\mu \cap \{u \ge u_\mu\}$. So we may estimate $I_3$ similar to $I_1, I_2$ as before
    \begin{align}\label{R23eqn7c} 
        I_3 \le 2C_{p,A}'(C' - 1)^2c^2\int_{B_{R+1} \setminus B_R} u^p \,dV_g \to 0^+
    \end{align}
    as $R \to \infty$.
    Therefore using the sharp decay estimate of $u$ in \eqref{R23eqn7} we have
    \begin{align}\label{R23eqn8}
    	& \int_{\Sigma_\mu \cap \{u \ge u_\mu\}} e^{-p\alpha_\lambda \dist_{\Bn}(0,x)} |\nabla_g (\varphi(\log u - \log u_\mu))|_g^2 \,dV_g \nn \\ &\le C'' e^{-(q-p)\dist_{\Bn}(0,\Sigma_\mu)}\int_{\Sigma_\mu \cap \{u \ge u_\mu\}} e^{-p\alpha_\lambda \dist_{\Bn}(0,x)} \varphi^2(\log u - \log u_\mu)^2 \,dV_g \nn \\ & \quad + 2I_1 + 2I_2 + I_3.
    \end{align}
    Recall from Lemma-\ref{R2lemmasubsol} we have $w(x) := \left[\sinh\left(\frac{1}{2}\dist_{\Bn}(0,x)\right)\right]^{-2(n-1)}$ is a subsolution to $-\Delta_g w \le 0$ in $\Bn \setminus B_{r_0}$ for some $r_0 = r_0(n) > 0$. Furthermore $\frac{|\nabla_g w|_g}{w} \ge c_n$ in $\Bn \setminus B_{r_0}$ for some $c_n > 0$. Since $\alpha := \frac{p\alpha_\lambda}{n-1} > 1$ we may apply the Hardy type inequality from \cite{DD}, Theorem-3.1 to write 
    \begin{align}\label{R2hardy1}
    	\left(\frac{|1-\alpha|}{2}\right)^2\int_{\Bn \setminus B_{r_0}} w^\alpha \frac{|\nabla_g w|_g^2}{w^2} v^2 \,dV_g \le \int_{\Bn \setminus B_{r_0}} w^\alpha |\nabla_g v|_g^2\,dV_g, \, \forall \, v \in C_c^\infty(\Bn \setminus B_{r_0}).
    \end{align}
    This combined with the observation $$w(x) \approx_{n,r_0} e^{-(n-1)\dist_{\Bn}(0,x)} \text{ in } \Bn \setminus B_{r_0}$$ gives the inequality
    \begin{align}\label{R2hardy2}
    	C_{n,\alpha_\lambda} \int_{\Bn} e^{-p\alpha_\lambda \dist_{\Bn}(0,x)} v^2 \,dV_g \le \int_{\Bn} e^{-p\alpha_\lambda \dist_{\Bn}(0,x)} |\nabla_g v|_g^2 \,dV_g, \, \forall \, v \in C_c^\infty(\Bn \setminus B_{r_0}).
    \end{align}
    We remark that the above inequality \eqref{R2hardy2} remains valid for all functions supported away from $B_{r_0}$ and with right decay property at infinity.
    \\ Then using inequality \eqref{R2hardy2} on RHS of \eqref{R23eqn8} with $v := \varphi(\log u - \log u_\mu)_+$ we get
    \begin{align}\label{R23eqn9}
    	& \left(1 - C'''e^{-(q-p)\alpha_{\lambda}\dist_{\Bn}(0,\Sigma_\mu)}\right) \int_{\Sigma_\mu \cap \{u \ge u_\mu\}} e^{-p\alpha_\lambda \dist_{\Bn}(0,x)} |\nabla_g (\varphi(\log u - \log u_\mu))|_g^2 \,dV_g \nn \\& \le \quad 2I_1 + 2I_2 + I_3.
    \end{align}
    Finally letting $R \to \infty$ from \eqref{R23eqn5} and \eqref{R23eqn7c} we get $2I_1 + 2I_2 + I_3 \to 0^+$, i.e.,
    \begin{align}\label{R23eqn10}
    	& \left(1 - C'''e^{-(q-p)\alpha_{\lambda}\dist_{\Bn}(0,\Sigma_\mu)}\right) \int_{\Sigma_\mu \cap \{u \ge u_\mu\}} e^{-p\alpha_\lambda \dist_{\Bn}(0,x)} |\nabla_g (\log u - \log u_\mu)|_g^2 \,dV_g \le 0.
    \end{align}
	Therefore $\nabla_g (\log u - \log u_\mu) \equiv 0$ in $\Sigma_\mu \cap \{u \ge u_\mu\}$ i.e., $u = u_\mu$ in $\Sigma_\mu \cap \{u \ge u_\mu\}$ for all $\mu$ sufficiently close to $1$. This completes the proof of the claim $u \le u_\mu$ in $\Sigma_\mu$ for all $\mu$ close to $1$.
	\\\\ \textbf{Step-II:} Let us denote $$M_u := \left\{\mu \in (-1,1) : u \le u_{\tilde{\mu}} \text{ in } \Sigma_{\tilde{\mu}}, \, \forall \, \mu \le \tilde{\mu} < 1 \right\}.$$ By the previous step $M_u \neq \emptyset$ and let $\mu_0 := \inf M_u$. To keep consistent with the previous step, wlog we may assume $\mu_0 > 0$, this may be achieved by making a change in coordinate frame by a Hyperbolic translation $\tau_{-te_1} \in \mathcal{M}(\Bn)$ (which fixes the points $-e_1$ and $e_1$) for $t > 0$ chosen sufficiently large if necessary and working with $v := (u \circ \tau_{-te_1})$ instead so that $\mu_0' =\inf M_v > 0$. 
	\\\\ We claim that $u \equiv u_{\mu_0}$ in $\Sigma_{\mu_0}$. Suppose to the contrary $u \not\equiv u_{\mu_0}$ in $\Sigma_{\mu_0}$, we will reach a contradiction on infimality of $\mu_0$ by showing that there is a $\overline{\epsilon} > 0$ such that $u \le u_{\mu_0 - \epsilon}$ in $\Sigma_{\mu_0 - \epsilon}$ for all $0 < \epsilon \le \overline{\epsilon}$. \\\\ The following argument is made exactly as in \cite{OSV}. Note that we have $u \le u_{\mu_0}$ in $\Sigma_{\mu_0}$ by continuity. Then by strong comparison principle Lemma-\ref{R2strongcomparison1}, we have either $u \equiv u_{\mu_0}$ or $u < u_{\mu_0}$ in each connected component of $\Sigma_{\mu_0} \setminus \C_u$, where, $\C_u$ is the critical point set of $u$. Note that by the gradient estimate of $u$ from below we know $\C_u$ is a compact set and furthermore from Lemma-\ref{R2negint1} we have $|\C_u| = 0$ ($\C_u$ is a set of Lebesgue measure zero). 
	\\ If $\Sigma_{\mu_0} \setminus \C_u$ is connected then by our assumption $u \not\equiv u_{\mu_0}$ in $\Sigma_{\mu_0}$ by using strong comparison in $\Sigma_{\mu_0} \setminus \C_u$ we have $u < u_{\mu_0}$ in $\Sigma_{\mu_0} \setminus \C_u$.
	\\ Else if $\Sigma_{\mu_0} \setminus \C_u$ has at least two connected components then compactness of $\C_u$ implies there is exactly one unbounded component which we denote by $\Omega_1$. Suppose, $u \equiv u_{\mu_0}$ in $\Omega_1$ then it is easy to see that $(\Sigma_{\mu_0} \setminus \Omega_1) \cup \sigma_{\mu_0}(\Sigma_{\mu_0} \setminus \Omega_1)$ contains a bounded connected component of $\Bn \setminus \C_u$ which we denote by $\omega$. Note that $\partial \omega \subset \C_u$. Recall from the regularity of $u$ from Remark-\ref{R2rmkreg} we have $|\nabla u|^{p-2}\nabla u \in W^{1,2}_{\text{loc}}(\Bn)$ so that this vector field is continuous and equals $0$ on $\partial \omega$. Since the equality $-\Delta_p^{\Bn} u = f(u)$ with $f(u) = \lambda u^{p-1} + u^{q-1}$ holds pointwise a.e. in $\omega$ we may integrate the expression in $\omega$ to write $$\int_{\omega} -\divg \left(\left(\frac{2}{1 - |x|^2}\right)^{n-p}|\nabla u|^{p-2}\nabla u \right) \,dx = \int_{\omega} f(u) \left(\frac{2}{1 - |x|^2}\right)^n \,dx.$$ Applying divergence theorem we have $$0 = -\int_{\partial \omega} \left(\frac{2}{1 - |x|^2}\right)^{n-p}|\nabla u|^{p-2}(\nabla u, \hat{n}) \,d\mathcal{H}^{n-1}  = \int_{\omega} f(u) \left(\frac{2}{1 - |x|^2}\right)^n \,dx > 0$$ which leads to a contradiction as $f(u) > 0$ in $\Bn$. Else if $u \equiv u_{\mu_0}$ in a bounded connected component $\Omega_2$ of $\Sigma_{\mu_0} \setminus \C_u$ then again $\Omega_2 \cup \sigma_{\mu_0}(\Omega_2)$ must contain a bounded connected component of $\Bn \setminus \C_u$ which is absurd as seen earlier. Thus by appealing to strong comparison we must have $u < u_{\mu_0}$ in each component of $\Sigma_{\mu_0} \setminus \C_u$ as well. 
	\\\\ Let us consider $\overline{R} > 0$ to be chosen sufficiently large later such that $\C_u \subset \subset B_{\overline{R}}$. Furthermore using the fact that $|\mathcal{C}_u| = 0$ (see Lemma-\ref{R2negint1}), we let $N := N_{\kappa}(\C_u)$ be a open neighborhood of $\C_u$ such that $|N| < \kappa$ has small measure. Then for $\overline{\epsilon}, \epsilon, \delta, \kappa > 0$ let us define the sets
	\begin{align}\label{R23eqn11}
		& \Sigma_{\mu_0-\epsilon}^{\overline{R}} := \Sigma_{\mu_0 - \epsilon} \setminus B_{\overline{R}}, \nn \\
		& S_\delta^\epsilon(\overline{R}) := \left((\Sigma_{\mu_0 - \epsilon}\setminus \Sigma_{\mu_0 + \delta}) \cap B_{\overline{R}}\right) \cup (N \cap \Sigma_{\mu_0 + \delta}), \nn \\ 
		& K_{\delta}(\overline{R}) := \overline{\Sigma_{\mu_0 + \delta} \cap B_{\overline{R}}} \cap (\Bn \setminus N)
	\end{align}
    so that we have 
    \begin{align} \label{R23eqn12}
    	\Sigma_{\mu_0 - \epsilon} = \Sigma_{\mu_0-\epsilon}^{\overline{R}} \cup S_\delta^\epsilon(\overline{R}) \cup K_{\delta}(\overline{R}).
    \end{align}
	We choose $\epsilon$ small enough such that $0 \not\in \Sigma_{\mu_0 - \epsilon}$ (i.e., $\Sigma_{\mu_0 - \epsilon}$ still remains to the right of the origin in $\Bn$) and $\delta$ small enough such that $K_\delta$ is non-empty. Since, $K_\delta(\overline{R}+1) := \overline{\Sigma_{\mu_0 + \delta} \cap B_{\overline{R} + 1}} \cap (\Bn \setminus N)$ is compact, by uniform continuity of $u$ and $u_\mu$ we have $u < u_{\mu_0 - \epsilon}$ in $K_\delta(\overline{R}+1)$ for each $\epsilon \le \overline{\epsilon}$ for some $\overline{\epsilon} > 0$.
	\\ Then as in the last step for $R > \overline{R}$, let $\varphi \in C_c^\infty(\Bn)$ be a non-negative test function such that $0 \le \varphi \le 1$ with $\varphi \equiv 1$ in $B_R$ and $\text{supp}(\varphi) \subset B_{R+1}$, $|\nabla_g \varphi|_g \le c$. We consider the test functions $\psi_1 := \varphi^2 u^{1-p}\left(u^p - u_{\mu_0 - \epsilon}^p\right)_+ \, \chi_{\Sigma_{\mu_0 - \epsilon}}$ and $\psi_2 := \varphi^2 u_{\mu_0 - \epsilon}^{1-p}\left(u^p - u_{\mu_0 - \epsilon}^p\right)_+ \, \chi_{\Sigma_{\mu_0 - \epsilon}}$ as before. Testing the PDE for $u$ with $\psi_1$ and the PDE for $u_{\mu_0 - \epsilon}$ with $\psi_2$, subtracting them and following the steps from Step-I we get
	\begin{align}\label{R23eqn13}
		& C_p \int_{\Sigma_{\mu_0 - \epsilon} \cap \{u \ge u_{\mu_0 - \epsilon}\}} \varphi^2 u^p \left(|\nabla_g \log u|_g + |\nabla_g \log u_{\mu_0 - \epsilon}|_g\right)^{p-2}|\nabla_g (\log u - \log u_{\mu_0 - \epsilon})|_g^2 \,dV_g \nn \\ 
		& \le Ce^{-(p\alpha_{\lambda} - (n-1))R} \nn \\ 
		& \quad + qp \int_{\Sigma_{\mu_0 - \epsilon}^{\overline{R}} \cap \{u \ge u_{\mu_0 - \epsilon}\}} \varphi^2 u^q (\log u - \log u_{\mu_0 - \epsilon})^2 \,dV_g \nn \\
		& \quad + qp \int_{S_\delta^\epsilon(\overline{R}) \cap \{u \ge u_{\mu_0 - \epsilon}\}} \varphi^2 u^q (\log u - \log u_{\mu_0 - \epsilon})^2 \,dV_g \nn \\ 
		& \le Ce^{-(p\alpha_{\lambda} - (n-1))R} \nn \\ 
		& \quad + 2qp \int_{\Sigma_{\mu_0 - \epsilon}^{\overline{R}} \cap \{u \ge u_{\mu_0 - \epsilon}\}} \varphi^2 u^q \eta^2(\log u - \log u_{\mu_0 - \epsilon})^2 \,dV_g \nn \\
		& \quad + 2qp \int_{S_\delta^\epsilon(\overline{R}+1) \cap \{u \ge u_{\mu_0 - \epsilon}\}} \varphi^2 u^q (1 - \eta)^2(\log u - \log u_{\mu_0 - \epsilon})^2 \,dV_g \nn \\ 
		& = Ce^{-(p\alpha_{\lambda} - (n-1))R} + J_1 + J_2
	\end{align}
    where, $\eta \in C^\infty(\Bn)$ is such that $0 \le \eta \le 1$ with $\eta \equiv 1$ in $\Bn \setminus B_{\overline{R} + 1}$ and $\eta \equiv 0$ in $B_{\overline{R}}$.
    \\ Let, 
    \begin{align}\label{R23eqn13a}
    	 J_0 := C_p \int_{\Sigma_{\mu_0 - \epsilon} \cap \{u \ge u_{\mu_0 - \epsilon}\}} u^p \left(|\nabla_g \log u|_g + |\nabla_g \log u_{\mu_0 - \epsilon}|_g\right)^{p-2}|\nabla_g (\log u - \log u_{\mu_0 - \epsilon})|_g^2 \,dV_g,
    \end{align}
    we claim that $J_0 \le 0$. Suppose to the contrary $J_0 > 0$.
    \\ Note that we have the gradient estimates \eqref{R2g1} and \eqref{R2g2} in $\Sigma_{\mu_0 - \epsilon}^{\overline{R}}$. Using the decay estimate of $u$ and the inequality \eqref{R2hardy2} followed by the gradient estimates of $u$ we may estimate $J_1$ as 
    \begin{align}\label{R23eqn14} 
    	J_1 &\le Ce^{-(q-p)\alpha_\lambda\overline{R}} \int_{\Sigma_{\mu_0 - \epsilon}^{\overline{R}} \cap \{u \ge u_{\mu_0 - \epsilon}\}} u^p \left(|\nabla_g \log u|_g + |\nabla_g \log u_{\mu_0 - \epsilon}|_g\right)^{p-2}|\nabla_g (\eta(\log u - \log u_{\mu_0 - \epsilon}))|_g^2 \,dV_g \nn 
    	\\ &\le 2Ce^{-(q-p)\alpha_\lambda\overline{R}} \int_{\Sigma_{\mu_0 - \epsilon}^{\overline{R}} \cap \{u \ge u_{\mu_0 - \epsilon}\}} u^p \left(|\nabla_g \log u|_g + |\nabla_g \log u_{\mu_0 - \epsilon}|_g\right)^{p-2}|\nabla_g (\log u - \log u_{\mu_0 - \epsilon})|_g^2 \,dV_g \nn \\
    	& \quad + 2C_Ae^{-(q-p)\alpha_\lambda\overline{R}} \int_{B_{\overline{R}+1} \setminus B_{\overline{R}}} u^p \,dV_g.
    \end{align}
    In the second line we have used the Cauchy-Schwartz inequality followed by the fact that $|\nabla_g \eta|_g \le c$ and is supported in $B_{\overline{R}+1} \setminus B_{\overline{R}}$ and $(\log u - \log u_{\mu_0 - \epsilon})_+$ is bounded from above in $\Sigma_{\mu_0-\epsilon}$.
    \\ Again in order to estimate $J_2$ we split it into two cases. In case $p > 2$ we may use the admissibility of $|\nabla_g u|_g^{p-2}$ as a weight for a 2-Poincar\'{e} inequality as stated in Remark-\ref{R2rmkweightedsobolev2} to write 
    \begin{align}\label{R23eqn15}
    	J_2 &\le C_u \int_{S_\delta^\epsilon(\overline{R} + 1) \cap \{u \ge u_{\mu_0 - \epsilon}\}} (1 - \eta)^2(\log u - \log u_{\mu_0 - \epsilon})^2 \,dV_g \nn 
    	\\ &\le C_u \Lambda(S_\delta^\epsilon(\overline{R} + 1)) \int_{S_\delta^\epsilon(\overline{R} + 1) \cap \{u \ge u_{\mu_0 - \epsilon}\}} |\nabla_g u|_g^{p-2} |\nabla_g ((1 - \eta)(\log u - \log u_{\mu_0 - \epsilon}))|_g^2\,dV_g \nn
    	\\ & \le 2C_u \Lambda(S_\delta^\epsilon(\overline{R} + 1)) \int_{S_\delta^\epsilon(\overline{R} + 1) \cap \{u \ge u_{\mu_0 - \epsilon}\}} |\nabla_g u|_g^{p-2} |\nabla_g (\log u - \log u_{\mu_0 - \epsilon})|_g^2\,dV_g \nn
    	\\ & \quad + 8C_u \Lambda(S_\delta^\epsilon(\overline{R} + 1)) \int_{S_\delta^\epsilon(\overline{R} + 1) \cap \{u \ge u_{\mu_0 - \epsilon}\}} |\nabla_g u|_g^{p-2} |\nabla_g \eta|_g^2(\log u - \log u_{\mu_0 - \epsilon})^2\,dV_g \nn 
    	\\ & := J_2' + J_2''.
    \end{align}
    where, $C_u = 2qp\left(\sup_{S_\delta^\epsilon} u^q\right)$ and $\Lambda(S_\delta^\epsilon(\overline{R} + 1))$ is the constant for the weighted 2-Poincar\'{e} inequality. Since $p > 2$ we have 
    \begin{align}\label{R23eqn16}
    	& J_2' \le 2C_u' \Lambda(S_\delta^\epsilon(\overline{R} + 1)) \int_{S_\delta^\epsilon(\overline{R} + 1) \cap \{u \ge u_{\mu_0 - \epsilon}\}} u^p|\nabla_g \log u|_g^{p-2} |\nabla_g (\log u - \log u_{\mu_0 - \epsilon})|_g^2\,dV_g \nn \\ & \le 2C_u' \Lambda(S_\delta^\epsilon(\overline{R} + 1)) \int_{S_\delta^\epsilon(\overline{R} + 1) \cap \{u \ge u_{\mu_0 - \epsilon}\}} u^p\left(|\nabla_g \log u|_g + |\nabla_g \log u_{\mu_0 - \epsilon}|_g\right)^{p-2} |\nabla_g (\log u - \log u_{\mu_0 - \epsilon})|_g^2\,dV_g
    \end{align}
    where, $C_u' = 2qp \left(\frac{\sup_{S_\delta^\epsilon} u^q}{\inf_{S_\delta^\epsilon} u^2}\right)$. 
    \\ Also, using the sharp decay, global gradient estimate from above of $u$ along with the fact that $|\nabla_g \eta|_g \le c$ and is supported in $B_{\overline{R}+1} \setminus B_{\overline{R}}$, followed by the fact that $(\log u - \log u_{\mu_0 - \epsilon})_+$ is bounded from above in the set, we have
    \begin{align}\label{R23eqn16a}
    	J_2'' \le C'' C_u \Lambda(S_\delta^\epsilon(\overline{R} + 1)) e^{-(p-2)\alpha_{\lambda}\overline{R}}\operatorname{Vol}_g(S_\delta^\epsilon(\overline{R}+1) \cap B_{\overline{R}+1} \setminus B_{\overline{R}}).
    \end{align}
    \\ Note that the constant $C_u'$ does not increase if we let $\epsilon, \delta \to 0^+$, thus remains uniformly bounded for all $\epsilon \le \overline{\epsilon}$ and $\delta < \overline{\delta}$. Furthermore from \eqref{R2weightedpoincare1} we have the Poincar\'{e} constant $\Lambda(S_\delta^\epsilon(\overline{R} + 1)) \to 0^+$ as $|S_\delta^\epsilon(\overline{R} + 1)| \to 0^+$.
    \\ In case $p < 2$ similar to the last case we may start by using the usual unweighted Poincar\'{e} inequality instead and the global gradient bound $$\left(|\nabla_g \log u|_g + |\nabla_g \log u_{\mu_0 - \epsilon}|_g\right)^{2-p} \le (2A)^{2-p}$$ to write 
    \begin{align}\label{R23eqn17}
    	J_2 &\le 2C_u \Lambda(S_\delta^\epsilon(\overline{R} + 1)) \int_{S_\delta^\epsilon \cap \{u \ge u_{\mu_0 - \epsilon}\}} |\nabla_g (\log u - \log u_{\mu_0 - \epsilon})|_g^2\,dV_g \nn 
    	\\ & \quad + 8C'''C_u \Lambda(S_\delta^\epsilon(\overline{R}+1)) \operatorname{Vol}_g(S_\delta^\epsilon(\overline{R}+1) \cap B_{\overline{R} + 1} \setminus B_{\overline{R}}) \nn 
    	\\& \le  2C_u' \Lambda(S_\delta^\epsilon(\overline{R} + 1)) \int_{S_\delta^\epsilon \cap \{u \ge u_{\mu_0 - \epsilon}\}} u^p\left(|\nabla_g \log u|_g + |\nabla_g \log u_{\mu_0 - \epsilon}|_g\right)^{p-2} |\nabla_g (\log u - \log u_{\mu_0 - \epsilon})|_g^2\,dV_g \nn 
    	\\ & \quad + 8C'''C_u \Lambda(S_\delta^\epsilon(\overline{R}+1)) \operatorname{Vol}_g(S_\delta^\epsilon(\overline{R}+1) \cap B_{\overline{R} + 1} \setminus B_{\overline{R}}).
    \end{align}
    where, the constant $C_u' = (2A)^{2-p}qp \left(\frac{\sup_{S_\delta^\epsilon} u^q}{\inf_{S_\delta^\epsilon} u^p}\right)$ does not increase as we let $\epsilon, \delta \to 0+$ and hence remains uniformly bounded as in the previous case.
    \\ Thus letting $R \to \infty$ in \eqref{R23eqn13} and combining with the fact that with $\overline{R}$ chosen large enough we have $J_1 \le J_0/4$ and in the estimates for $J_2$ we can choose $\epsilon \le \overline{\epsilon}$, $\delta$ and $\kappa$ small enough s.t., $2C_u'\Lambda(S_\delta^\epsilon(\overline{R} + 1)) < 1/4$, $J_2'' < J_0/4$ we get
    \begin{align}\label{R23eqn18}
    	J_0/4 = \frac{C_p}{4} \int_{\Sigma_{\mu_0 - \epsilon} \cap \{u \ge u_{\mu_0 - \epsilon}\}} u^p \left(|\nabla_g \log u|_g + |\nabla_g \log u_{\mu_0 - \epsilon}|_g\right)^{p-2}|\nabla_g (\log u - \log u_{\mu_0 - \epsilon})|_g^2 \,dV_g \le 0
    \end{align}
    which leads to a contradiction. Thus we must have $J_0 \le 0$ to begin with.
    \\ This in turn implies as in Step-I, $u \le u_{\mu_0 - \epsilon}$ in $\Sigma_{\mu_0 - \epsilon}$ for each $\epsilon \le \overline{\epsilon}$. But this contradicts the minimality of $\mu_0$, completing the proof.
    \\\\ \textbf{Step-III:} Finally applying another Hyperbolic translation $\tau_{te_1} \in \mathcal{M}(\Bn)$ for some $t \in \mathbb{R}$ we may choose a coordinate frame s.t., $\mu_0 = 0$. Thus we have $u \equiv (u \circ \sigma_{0,e_1}$) in $\Sigma_{0}(e_1)$. Similarly repeating this in $e_j$-directions for $j = 2, \cdots, n$ we also have $u \equiv (u \circ \sigma_{0,e_j}$) in $\Sigma_{0}(e_j)$. Owing to the fact that $u$ is non-decreasing in direction of the moving plane we must have $u(0) = \sup_{\Bn} u = M$ (say). Since, $u$ is symmetric about the $n$ orthogonal planes passing through $0 \in \Bn$, it must also have point symmetry about $0$, meaning $u(x) = u(-x)$ by successively reflecting along these $n$ planes. We claim that for any direction $\overline{e} \in \mathbb{S}^{n-1}$, the moving plane in the direction of $\overline{e}$ must stop only after it reaches the origin $0 \in \Bn$. Otherwise, if we have $u \equiv (u \circ \sigma_{\mu, \overline{e}})$ in $\Sigma_{\mu}(\overline{e})$ for some $\mu \neq 0$, i.e., $T_{\mu}(\overline{e})$ stops at an optimal position before reaching $0$, by reflecting $0$ about $T_\mu(\overline{e})$ we must have $u \equiv M$ along the line segment joining $0$ and its reflection $O_1 = \sigma_{\mu,\overline{e}}(0)$. Again due to the point symmetry about $0$ we must have $u(-O_1) = M$. By reflecting $-O_1$ about $T_\mu(\overline{e})$ we have $u \equiv M$ on the line segment joining $-O_1$ and its reflection $O_2 = \sigma_{\mu, \overline{e}}(-O_1)$. If we denote $\dist_{\Bn}(0,O_1) = 2\ell$, since reflections are isometries, successively iterating this process with $O_k := \sigma_{\mu, \overline{e}}(-O_{k-1})$ we get $\dist_{\Bn}(0,O_k) = 2k\ell$. Therefore, $u \equiv M$ along the entire line passing through $0$ in direction $\overline{e}$, which contradicts the decay of $u$. Therefore, $u \equiv (u \circ \sigma_{0, \overline{e}})$ for any $\overline{e} \in \mathbb{S}^{n-1}$, which implies $u$ is a radially symmetric function in $\Bn$ with respect to the origin $0$. To see that $u$ is strictly decreasing in radial direction (i.e., $\mathcal{C}_u = \{0\}$), we may either look at the ordinary differential equation satisfied by $u$ or argue as in beginning of Step-II to conclude that $\Bn \setminus \mathcal{C}_u$ cannot have a bounded connected component since, level sets of $u$ are spheres. This completes the proof of Hyperbolic symmetry of $u$. Finally we note that if we denote $u(x) = \Phi(\dist_{\Bn}(0,x))$ then $\Phi'(t) < 0$ for all $t > 0$ and by standard regularity argument $\Phi(t)$ is smooth for $t > 0$. 
    \end{proof}
	
	\section{Strong comparison for Hyperbolic $p$-Laplace equation}\label{R2strongcopmarison}
	In this section we establish the Sobolev regularity of $|\nabla_g u|_g^{p-2}\nabla_g u$ and local integrability of $|\nabla_g u|_g^{-(p-1)r}$ for $r < 1$, of positive solutions $u \in \mathcal{D}^{1,p}(\Hn) \cap C_{\text{loc}}^{1,\alpha}(\Hn)$ of the equation \begin{align}\label{R2eqn41} 
		-\Delta_p^{\Hn} u &= f(u) 
	\end{align} 
    where, $f$ is a positive locally Lipschitz function such that $f(t) > 0$ for $t > 0$ and the set of critical points of $u$ is compact. These estimates play a crucial role in proving a strong comparison theorem when $p > 2$. We will proceed in similar lines of \cite{DS1}, \cite{DS2}. We remark that these results have also been proved in a more general homogeneous compact Riemannian manifolds with boundary in case $p > 2$ in \cite{ADG2}. In the prior sections we saw that when $f(t) = \lambda t^{p-1} + t^{q-1}$ for $p < q \le p^\ast$ and $\displaystyle 0 \le \lambda < \lambda_{\text{max}} :=  \left(\frac{n-1}{p}\right)^p$, the positive finite energy solutions have non-vanishing gradient in complement of a large ball. For the sake of simplicity of calculations we will be using the upper half-space model $(\Hn, g)$ of the Hyperbolic space.
	\\\\ We begin by recalling that the equation \eqref{R2eqn41} is invariant under the action of the Mobius transformations $\mathcal{M}(\Hn)$ (isometries of $\Hn$). Let, $\tau \in \mathcal{M}(\Hn)$ then $u_\tau := (u \circ \tau) \in \mathcal{D}^{1,p}(\Hn)$ also solves \begin{align}\label{R2eqn42} 
		-\Delta_p^{\Hn} u_\tau = f(u_\tau).
	\end{align}
	\\ We will linearize the equation \eqref{R2eqn41} along the flows of the Hyperbolic translations. Note that for $t > 0$, $\tau_t \in \mathcal{M}(\Hn)$ given by $\tau_t(x) := tx$ is a translation in the upper half-space model of the Hyperbolic space that fixes $0$ and $\infty$. Note that by classical regularity theory $u \in C^{2,\sigma}$ in $\Hn \setminus \C_u$ where, $\C_u := \{x \in \Hn: \nabla_g u (x) = 0\}$ is the critical point set. Let $\varphi \in C_c^\infty(\Hn)$ be such that $\text{supp}(\varphi) \subset \Hn \setminus \C_u$. Testing equation \eqref{R2eqn42} with $\varphi$ and differentiating at $t = 1$ we have 
	\begin{align}\label{R2eqn43} 
		\left.\frac{d}{dt}\right\vert_{t = 1} \int_{\Hn} |\nabla_g (u \circ \tau_t)|_g^{p-2} g(\nabla_g (u \circ \tau_t),\nabla_g \varphi) \,dV_g = \left.\frac{d}{dt}\right\vert_{t = 1} \int_{\Hn} f(u \circ \tau_t) \varphi \,dV_g. \end{align}
	\\ Note that we have successively for each $j = 1, \cdots, n$ 
	\begin{align}\label{R2id1} 
		& \left.\frac{d}{dt}\right\vert_{t = 1} (u \circ \tau_t) = (x,\nabla u) := V_n(u). \\ \label{R2id2} & \frac{\partial (u \circ \tau_t)}{\partial x_j} = tu_{x_j}(tx) \\ \label{R2id3} & \left.\frac{d}{dt}\right\vert_{t = 1} \frac{\partial (u \circ \tau_t)}{\partial x_j} = u_{x_j} + \left(x,\nabla u_{x_j}\right) = \frac{\partial V_n(u)}{\partial x_j} \\ \label{R2id4} & \left.\frac{d}{dt}\right\vert_{t = 1} |\nabla (u \circ \tau_t)|^2 = 2(\nabla u, \nabla V_n(u)). 
	\end{align}
	\\ Therefore, \eqref{R2eqn43} becomes 
	\begin{align}\label{R2eqn44} 
		& \int_{\Hn} x_n^{p-n}\left(|\nabla u|^{p-2}(\nabla V_n(u), \nabla \varphi) + (p-2)|\nabla u|^{p-4}(\nabla V_n(u), \nabla u)(\nabla u, \nabla \varphi)\right)\,dx \nn \\& =  \int_{\Hn} f'(u) V_n(u) \varphi x_n^{-n}\,dx, \, \forall \, \varphi \in C_c^\infty(\Hn \setminus \C_u). 
	\end{align}
	\\ Keeping consistent with the notations in \cite{DS1}, we define the linearized PDE corresponding to \eqref{R2eqn41} to be, $L_u(v,\varphi) = 0$ for $\varphi \in C_c^\infty(\Hn)$ where 
	\begin{align} \label{R2eqn45} 
		L_u(v,\varphi) &:= \int_{\Hn} |\nabla_g u|_g^{p-2}g(\nabla_g v, \nabla_g \varphi) + (p-2)|\nabla_g u|_g^{p-4}g(\nabla_g v, \nabla_g u)g(\nabla_g u, \nabla_g \varphi)\,dV_g \nn \\ & \quad - \int_{\Hn} f'(u) v \varphi \,dV_g. 
	\end{align} 
	\\ Similarly considering horizontal translations as well in upper half-space  (Euclidean translations $x \mapsto x + te_j$ in $e_j$-direction for $j=1, \cdots, n-1$, which are also isometries in the upper half-space model of the Hyperbolic space) we have 
	\begin{align}\label{R2eqn46}
		 L_u(V_j(u), \varphi) = 0 
	\end{align} 
    $\forall \, \varphi \in C_c^\infty(\Hn \setminus \C_u)$ where, 
    \begin{align} \label{R2vf1} 
    	V_j(u) &:= u_{x_j} \text{ for } j = 1, \cdots , n-1 \nn \\ V_n(u) &= (x,\nabla u).
    \end{align} 
	By density argument the above equation remains valid $\varphi\in W^{1,2}_{\text{loc}}(\Hn)$ with compact support in $\Hn \setminus \C_u$. We note that the vector field $V(u) := (V_1(u), \cdots, V_n(u))^T$ is given by 
	\begin{align}\label{R2eqn47} 
		V(u) = A(x)\nabla u
	\end{align} 
    where, $\nabla u = (u_{x_1}, \cdots, u_{x_n})^T$, $A(x) := [e_1, \cdots, e_{n-1}, x]^T$ (where, $e_1, e_2, \cdots, e_n, x$ are the column vectors of $[e_1, \cdots, e_{n-1}, x]$) is an invertible matrix with its inverse $A(x)^{-1} = [e_1, \cdots , e_{n-1},\overline{x}]^T$ where $\overline{x} = \left(-\frac{x_1}{x_n}, \cdots , -\frac{x_{n-1}}{x_n}, \frac{1}{x_n}\right)^T$. We also have $\lVert A(x)^{-1} \rVert_2^{-1} |\nabla u| \le |V(u)| \le \lVert A(x) \rVert_2 |\nabla u|$ where, $\lVert A(x) \rVert_2^2 = n-1 + |x|^2$ and $\lVert A(x)^{-1} \rVert_2^2 = n-1 + \frac{|x|^2 - x_n^2 + 1}{x_n^2}$. \\ Keeping the linearized equation \eqref{R2eqn46} in mind we proceed to state the regularity results for $V(u)$. The proofs are verbatim adaptations from \cite{DS1} and \cite{DS2}, some of which we present here and summarize the rest for the sake of completeness. See Theorem-2.2 in \cite{DS1} for the corresponding Euclidean analogue of the following lemma.
    \begin{lemma}\label{R2W12reg} 
    	Let, $1 < p < \infty$ and $u \in \mathcal{D}^{1,p}(\Hn) \cap C^{1,\sigma}_{\text{loc}}(\Hn)$ be a solution of \eqref{R2eqn41} such that the critical point set $\C_u := \left\{x \in \Hn: \nabla_g u = 0\right\}$ is compact and contained in a bounded domain $\Omega \subset \Hn$. Let us denote by $Z_k := \left\{x \in \Hn : V_k(u)(x) = 0\right\}$ for each $k = 1, 2, \cdots, n$. Then for each compact set $E$ such that $\C_u \subset \subset E \subset \subset \Omega$ we have 
    \begin{align}
    	\label{R2W12regeq1} \sup_{x \in \Omega} \int_{E \setminus Z_k} \frac{|\nabla_g u|_g^{p-2} \, |\nabla_g V_k(u)|_g^2}{|V_k(u)|^\beta \, \dist_{\Hn}(x,y)^\gamma} \,dV_g(y) < C, \text{ for } k = 1,2, \cdots, n 
    \end{align} 
    for some positive constant $C = C(u, \beta, \gamma, E)$ (depending on $u$) where, $\beta, \gamma$ are non-negative real numbers such that $0 \le \beta < 1$ and $0 \le \gamma < n-2$ when $n \ge 3$ ($\gamma = 0$ when $n=2$). This in turn implies \begin{align}\label{R2W12regeq2} 
    	\sup_{x \in \Omega} \int_{E \setminus \C_u} \frac{|\nabla_g u|_g^{p-2-\beta} \, |\nabla_g^2 u|_g^2}{\dist_{\Hn}(x,y)^\gamma} \,dV_g(y) < C. 
    \end{align}
    \end{lemma} 

    \begin{proof} 
    	
    	 Wlog we may assume $x \in E$. Indeed, suppose we prove $$\sup_{x \in E} \int_{E \setminus Z_k} \frac{|\nabla_g u|_g^{p-2} \, |\nabla_g V_k(u)|_g^2}{|V_k(u)|^\beta \, d_{\Hn}(x,y)^\gamma} \,dV_g(y) := K(\beta, \gamma, E) < C.$$ 
    	Then splitting between the two cases $x \in E_\delta := \left\{x \in \Omega: \dist_{\Bn}(x,E) < \delta\right\}$ with $\dis \delta < \frac{1}{2}\dist_{\Bn}(E,\partial \W)$ and $x \in \Omega \setminus E_\delta$ we have $$\sup_{x \in \Omega} \int_{E \setminus Z_k} \frac{|\nabla_g u|_g^{p-2} \, |\nabla_g V_k(u)|_g^2}{|V_k(u)|^\beta \, d_{\Hn}(x,y)^\gamma} \,dV_g(y) \le K(\beta, \gamma, E_\delta) + \frac{1}{\delta^{\gamma}}K(\beta, 0, E).$$
    	\\ Let $\dis G_\ve(s) = \begin{cases} 0 & \text{ when } |s| \le \ve \\ s & \text{ when } |s| \ge 2\ve \end{cases}$ and linear continuous when $\ve < |s| < 2\ve$. Then $G_\ve$ is a Lipschitz function with $|G_\ve'| \le 2$. Furthermore we have $\dis \left( G_\ve'(s) - \beta \frac{G_\ve(s)}{s} \right) \ge 0$ for all $s \in \R$ and $\beta < 1$.
    	\\\\ \textbf{Case-1:} Let $x \in E \cap \mathcal{C}_u$ and let $\phi \in C_c^\infty(\Omega)$ be a non-negative test function such that $\phi \equiv 1$ in a neighborhood containing $E_\delta := \left\{y \in \Omega : \dist_{\Bn}(y,E) < \delta\right\}$ of $E$ and $\text{supp}(\phi) \subset \W$ with $\delta < \frac{1}{2}\dist_{\Hn}(E,\partial \Omega)$. Consider the test function $\dis \psi_{\ve,x}(y) := \frac{G_\ve(V_k(u)(y))}{|V_k(u)(y)|^\beta} \frac{\phi(y)}{\dist_{\Hn}(x,y)^\gamma}$, for the linearized equation \eqref{R2eqn46}.
    	Then we have \begin{align} \label{R2s1} 
    		& \int_{\W\setminus Z_k} |\nabla_g u|_g^{p-2}|\nabla_g V_k(u)|_g^{2} \left(G_\ve'(V_k(u)) - \beta\frac{G_\ve(V_k(u))}{V_k(u)}\right) \frac{\phi(y)}{|V_k(u)|^\beta \dist_{\Hn}(x,y)^\gamma} \,dV_g(y) \nn 
    		\\& + (p-2) \int_{\W\setminus Z_k} |\nabla_g u|_g^{p-4}g\left(\nabla_g u , \nabla_g V_k(u)\right)^{2} \left(G_\ve'(V_k(u)) - \beta\frac{G_\ve(V_k(u))}{V_k(u)}\right) \frac{\phi(y)}{|V_k(u)|^\beta \dist_{\Hn}(x,y)^\gamma} \,dV_g(y) \nn 
    		\\& + \int_{\W\setminus E_\delta} |\nabla_g u|_g^{p-2}g\left(\nabla_g V_k(u), \nabla_g \phi\right) \frac{G_\ve(V_k(u))}{|V_k(u)|^\beta\dist_{\Hn}(x,y)^\gamma}\,dV_g(y) \nn 
    		\\& + (p-2)\int_{\W\setminus E_\delta} |\nabla_g u|_g^{p-4}g\left(\nabla_g V_k(u), \nabla_g u\right) g(\nabla_g u , \nabla_g \phi) \frac{G_\ve(V_k(u))}{|V_k(u)|^\beta\dist_{\Hn}(x,y)^\gamma}\,dV_g(y) \nn 
    		\\& + \int_{\W\setminus Z_k} |\nabla_g u|_g^{p-2}g\left(\nabla_g V_k(u), \nabla_g \dist_{\Hn}(x,y)^{-\gamma} \right) \frac{G_\ve(V_k(u)) \phi }{|V_k(u)|^\beta}\,dV_g(y) \nn
    		\\& + (p-2)\int_{\W\setminus Z_k} |\nabla_g u|_g^{p-4}g\left(\nabla_g V_k(u), \nabla_g u\right) g(\nabla_g u , \nabla_g \dist_{\Hn}(x,y)^{-\gamma}) \frac{G_\ve(V_k(u)) \phi}{|V_k(u)|^\beta}\,dV_g(y) \nn 
    		\\ &= \int_{\W} f'(u)V_k(u) \psi_{\ve,x}(y)\,dV_g(y). \end{align}
    	\\ Note that the RHS of \eqref{R2s1} is bounded by \begin{align} \label{R2s2} \lVert f'(u) |V_k(u)|^{1 - \beta} \rVert_{L^\infty(\W)} \int_{\W} \frac{\phi(y)}{\dist_{\Hn}(x,y)^\gamma}\,dV_g(y) \le C \end{align} since, $\gamma < n-2$.
    	\\ Now, in the third and fourth integral on LHS of \eqref{R2s1} we have $\dist_{\Hn}(x,y)^{-\gamma} \le \delta^{-\gamma}$ in $\W \setminus E_\delta$. Also, $|\nabla_g u|_g^{p-2}|\nabla_g V_k(u)|_g \le C_1$ since, $\nabla_g u \neq 0$ in $\W \setminus E_\delta$ and $u \in C^{2,\alpha}(\Omega \setminus E_\delta)$ by standard elliptic regularity theory. Therefore, the two integrals are bounded by \begin{align} \label{R2s3} (p-1)\int_{\W \setminus E_\delta} \frac{|\nabla_g u|^{p-2} |\nabla_g V_k(u)|_g|\nabla_g \phi|}{\dist_{\Hn}(x,y)^\gamma} \frac{|G_\ve(V_k(u))|}{|V_k(u)|^\beta} \,dV_g(y) \le C_1'\delta^{-\gamma}. \end{align}
    	\\ The fifth and sixth integrals on LHS of \eqref{R2s1} is bounded by \begin{align}\label{R2s4} (p-1) \int_{\W \setminus Z_k} \frac{|\nabla_g u|^{p-2} |\nabla_g V_k(u)|_g}{\dist_{\Hn}(x,y)^{\gamma + 1}} \frac{|G_\ve(V_k(u))| \phi}{|V_k(u)|^\beta} \,dV_g(y). \end{align} Now, using Young's inequality for any $\sigma > 0$ we may write \begin{align}\label{R2s5} \frac{|\nabla_g V_k(u)|_g}{\dist_{\Hn}(x,y)^{\gamma + 1}} \le \sigma \frac{|\nabla_g V_k(u)|_g^2}{|V_k(u)|\dist_{\Hn}(x,y)^{\gamma}} + \frac{1}{4\sigma} \frac{|V_k(u)|}{\dist_{\Hn}(x,y)^{\gamma+2}}. \end{align}
    	Therefore, using \eqref{R2s5} in \eqref{R2s4} and the fact that $|G_\ve(V_k(u))| \le |V_k(u)| \le C|\nabla_g u|$ in $\Omega$ we have 
    	\begin{align}\label{R2s6} & (p-1) \int_{\W \setminus Z_k} \frac{|\nabla_g u|^{p-2} |\nabla_g V_k(u)|_g}{\dist_{\Hn}(x,y)^{\gamma + 1}} \frac{|G_\ve(V_k(u))| \phi}{|V_k(u)|^\beta} \,dV_g(y) \nn \\ \le & \quad \min\{1, p-1\}\sigma \int_{\W \setminus Z_k} \frac{|\nabla_g u|^{p-2} |\nabla_g V_k(u)|_g^2}{\dist_{\Hn}(x,y)^{\gamma}} \frac{|G_\ve(V_k(u))|}{|V_k(u)|} \frac{\phi}{|V_k(u)|^{\beta}} \,dV_g(y) \nn \\ & \quad  + \frac{(p-1)^2}{4\min\{1, p-1\} \sigma} \int_{\W \setminus Z_k} \frac{|\nabla_g u|^{p-\beta}\phi}{\dist_{\Hn}(x,y)^{\gamma + 2}} \,dV_g(y) \end{align} where, the last integral on RHS of \eqref{R2s6} is bounded as $\gamma < n-2$.
    	\\ Considering the two cases $p > 2$ and $1 < p < 2$, using Cauchy-Schwartz inequality we have $$\min\{1 , p-1\} |\nabla_g V_k(u)|_g^{2} \le |\nabla_g V_k(u)|_g^{2} + (p-2) |\nabla_g u|_g^{-2} g\left(\nabla_g u, \nabla_g V_k(u)\right)^2.$$ So the sum of first two integrals on LHS of \eqref{R2s1} is bounded from below by \begin{align} \label{R2s7} \min\{1 , p-1\} \int_{\W\setminus Z_k} |\nabla_g u|_g^{p-2}|\nabla_g V_k(u)|_g^{2} \left(G_\ve'(V_k(u)) - \beta\frac{G_\ve(V_k(u))}{V_k(u)}\right) \frac{\phi(y)}{|V_k(u)|^\beta \dist_{\Hn}(x,y)^\gamma} \,dV_g(y). \end{align}
    	\\ Choosing $\sigma$ small s.t., $1 - \beta - \sigma > 0$ and absorbing the first integral on the RHS on \eqref{R2s6} in \eqref{R2s7} and combining with the bounds established above in \eqref{R2s1}, 
    	\begin{align} \label{R2s8} \int_{\W\setminus Z_k} |\nabla_g u|_g^{p-2}|\nabla_g V_k(u)|_g^{2} \left(G_\ve'(V_k(u)) - (\beta + \sigma)\frac{G_\ve(V_k(u))}{V_k(u)}\right) \frac{\phi(y)}{|V_k(u)|^\beta \dist_{\Hn}(x,y)^\gamma} \,dV_g(y) \le C_0.\end{align}
    	Now, $\dis \left(G_\ve'(V_k(u)) - (\beta + \sigma)\frac{G_\ve(V_k(u))}{V_k(u)}\right) \to 1 - (\beta + \sigma)$ in $\W \setminus Z_k$ as $\ve \to 0$, so by applying Fatou's lemma in \eqref{R2s8},
    	\begin{align}\label{R2s9} 
    		\int_{E\setminus Z_k} \frac{|\nabla_g u|_g^{p-2}|\nabla_g V_k(u)|_g^{2}}{|V_k(u)|^\beta \dist_{\Hn}(x,y)^\gamma} \,dV_g(y) \le \int_{\W\setminus Z_k} \frac{|\nabla_g u|_g^{p-2}|\nabla_g V_k(u)|_g^{2} \phi(y)}{|V_k(u)|^\beta \dist_{\Hn}(x,y)^\gamma} \,dV_g(y) \le C_0 
    	\end{align} 
    	where, $C_0$ does not depend on $x \in E\cap \mathcal{C}_u$. Also, since $\nabla_g V_k(u) = 0$ a.e. in $Z_k \setminus \mathcal{C}_u$ it follows that 
    	\begin{align} \label{R2s10} 
    		\int_{E\setminus \mathcal{C}_u} \frac{|\nabla_g u|_g^{p-2}|\nabla_g V_k(u)|_g^{2}}{|V_k(u)|^\beta \dist_{\Hn}(x,y)^\gamma} \,dV_g(y) \le C_0. 
    	\end{align}
    	\\ Thus the conclusion of \eqref{R2W12regeq2} now easily follows from \eqref{R2eqn47}.
    	\\\\ \textbf{Case-2:} Let, $x \in E \setminus \mathcal{C}_u$. Then choose $\ve_0 > 0$ s.t., $B_{2\ve_0}(x) \cap \mathcal{C}_u = \emptyset$. For $\ve < \ve_0$ choose non-negative test function $\phi_{\ve,x} \in C_c^\infty(\W)$ s.t., $\phi_{\ve,x} \equiv 0$ in $B_\ve(x)$ and $\phi_{\ve,x} \equiv 1$ in $E_\delta \setminus B_{2\ve}(x)$, $|\nabla_g \phi_{\ve,x}|_g \le \frac{C}{\ve}$ in $B_{2\ve}(x) \setminus B_\ve(x)$ and $|\nabla_g \phi_{\ve,x}|_g \le C$ in $\W \setminus B_{2\ve}(x)$.
    	\\ Then using the test function $\dis \psi_{\ve,x} := \frac{G_\ve(V_k(u)(y))}{|V_k(u)(y)|^\beta} \frac{\phi_{\ve,x}(y)}{\dist_{\Hn}(x,y)^\gamma}$ in \eqref{R2eqn46} and proceeding as before we have 
    	\begin{align}\label{R2s11} & \int_{\W\setminus Z_k} |\nabla_g u|_g^{p-2}|\nabla_g V_k(u)|_g^{2} \left(G_\ve'(V_k(u)) - (\beta + \sigma)\frac{G_\ve(V_k(u))}{V_k(u)}\right) \frac{\phi_{\ve,x}(y)}{|V_k(u)|^\beta \dist_{\Hn}(x,y)^\gamma} \,dV_g(y) \nn 
    		\\ & \le C_0' + C' \int_{B_{2\ve}(x)\setminus B_\ve(x)} |\nabla_g u|_g^{p-2}|\nabla_g V_k(u)|_g \frac{|G_{\ve}(V_k(u))||\nabla_g \phi_{\ve,x}(y)|_g}{|V_k(u)|^\beta \dist_{\Hn}(x,y)^\gamma} \,dV_g(y). \end{align}
    	\\ Then using the fact that $u \in C^{2,\alpha}(B_{2\ve_0}(x))$ so that $\sup_{B_{2\ve_0}(x)} |\nabla_g u|_g^{p-2}|\nabla_g V_k(u)|_g|V_k(u)|^{1 - \beta} = a(x)$ is a bounded quantity independent of $\ve > 0$, and noting that $V_g(B_{2\ve}(x)) \approx \ve^n$ we get 
    	\begin{align} \label{R2s12} \int_{B_{2\ve}(x)\setminus B_\ve(x)} |\nabla_g u|_g^{p-2}|\nabla_g V_k(u)|_g \frac{|G_{\ve}(V_k(u))||\nabla_g \phi_{\ve,x}(y)|_g}{|V_k(u)|^\beta \dist_{\Hn}(x,y)^\gamma} \,dV_g(y) \le C'a(x) \frac{\ve^n}{\ve^{\gamma + 1}} \to 0^+ \end{align} as $\ve \to 0$. Thus using Fatou's lemma as before in \eqref{R2s11} we have \eqref{R2W12regeq1} in this case as well. \end{proof}
    
    \begin{remark}\label{R2rmkreg} Note that if $u$ be as in Lemma-\ref{R2W12reg}, by Stampacchia's lemma $|\nabla_g u|_g^{p-2}\nabla_g u \in W^{1,2}_{\text{loc}}(\W)$. Indeed if we set $G_\ve$ as in Lemma-\ref{R2W12reg} then $$ \frac{\partial}{\partial x_j} G_\ve\left(|\nabla u|^{p-2}u_{x_i}\right) = G_\ve'\left(|\nabla u|^{p-2}u_{x_i}\right)\left[|\nabla u|^{p-2}u_{x_ix_j} + (p-2)|\nabla u|^{p-4}(\nabla u, \nabla u_{x_j})u_{x_i}\right]$$ a.e. in $\Omega$. Also, applying previous lemma with $\beta = \gamma = 0$ we get $$\lVert G_\ve\left(|\nabla u|^{p-2}u_{x_i}\right) \rVert_{W^{1,2}(E)} \le C, \, \forall \, \ve \text{ small }$$ where, $E \subset \subset \W$ and $C$ depends only on $E$. Then by compact embedding $W^{1,2}(E) \subset L^2(E)$ we may extract a sequence $\ve_m \to 0$ s.t., $G_{\ve_m}\left(|\nabla u|^{p-2}u_{x_i}\right) \to v$ in $L^2(E)$ and pointwise a.e. in $E$ for some $v \in W^{1,2}(E)$. However, $G_{\ve_m}\left(|\nabla u|^{p-2}u_{x_i}\right) \to |\nabla u|^{p-2}u_{x_i}$ a.e. in $E$, thus $|\nabla u|^{p-2}u_{x_i} \equiv v \in W^{1,2}(E)$.
    \end{remark}
 
    \begin{remark}\label{R2rmklineqn} As a consequence we have $V_j(u)$ is a weak solution to the linearized equation \eqref{R2eqn46}, for $j = 1, \cdots, n$. The Sobolev regularity of $|\nabla_g u|_g^{p-2}\nabla_g u$ can be used to justify the differentiation under integral sign, in particular where $\text{supp}(\varphi)$ intersects the critical point set $\C_u$. Therefore we have 
    	\begin{align}\label{R2eqnregimp}
    		L_u(V_j(u), \varphi) = 0 
    	\end{align} 
    	$\forall \, \varphi \in C_c^\infty(\Hn)$ and $j = 1, \cdots, n$.
    \end{remark}
    \, \\ With the aid of Lemma-\ref{R2W12reg} one may then proceed exactly as in the proof of Theorem-2.3 in \cite{DS1} for the Euclidean case, to prove the following lemma concerning the integrability of negative exponents of $|\nabla_g u|_g$ along with a potential.
    \begin{lemma}\label{R2negint1} Let $u$ be as in Lemma-\ref{R2W12reg}. Then \begin{align}\label{R2int2} \sup_{x\in \Omega} \int_{\Omega} |\nabla_g u|_g^{-(p-1)r}\dist_{\Hn}(x,y)^{-\gamma} \,dV_g(y) < C \end{align} where, $r,\gamma$ are non-negative real numbers such that $0 \le r < 1$, $0 \le \gamma < n-2$ when $n \ge 3$ ($\gamma = 0$ when $n=2$). In particular this implies $|\C_u| = 0$.
     \end{lemma}
    \, \\ The following inequalities are a direct consequence of corresponding Euclidean inequalities of Theorem-3.1 in \cite{DS1} and the conformal relation of the Hyperbolic volume element with the Euclidean volume element which we state here for the sake of clarity.
    
    \begin{lemma}\label{R2weightedpoincarelemma1}
    Let, $\tilde{\Omega}$ be a bounded domain in $\Hn$ and let $\Omega \subset \tilde{\Omega}$ be a subdomain. Suppose $\rho \in L^1(\tilde{\Omega})$ to be a positive weight function satisfying 
    \begin{align}\label{R2negintcond1}
    		\int_{\tilde{\Omega}} \rho^{-t}(y)\dist_{\Hn}(x,y)^{-\gamma} \,dV_g(y) \le C(\tilde{\Omega}), \, \forall \, x \in \tilde{\Omega}
    \end{align} for some positive number $t > 0$ and $0 \le \gamma < n-2$.
    Let, $m$ be an exponent satisfying $m > 1 + \frac{1}{t}$ and $m > \frac{n-\gamma}{t}$. 
    \begin{enumerate}
    \item  If $m$ satisfies $m < n\left(1 + \frac{1}{t}\right) - \frac{\gamma}{t}$ and define the exponent $m^\sharp$ as $\frac{1}{m^\sharp} := \frac{1}{m} - \frac{1}{n} + \frac{1}{mt}\left(1 - \frac{\gamma}{n}\right)$ then the following weighted Sobolev inequality is valid
    \begin{align}\label{R2weightedsobolev1}
    	\lVert v \rVert_{L^{m^\sharp}(\Omega)} \le C \lVert \nabla_g v \Vert_{L^m(\Omega, \rho)}, \, \forall \, v \in W_0^{1,m}(\Omega, \rho)
    \end{align}
    where, the constant $C = C(n,m,t,\gamma, \rho,\tilde{\Omega})$.
    \item If $m$ satisfies $m = n\left(1 + \frac{1}{t}\right) - \frac{\gamma}{t}$ then we have 
    \begin{align}\label{R2weightedineq2}
    	\lVert v \rVert_{L^{q}(\Omega)} \le C_q \lVert \nabla_g v \Vert_{L^m(\Omega, \rho)}, \, \forall \, v \in W_0^{1,m}(\Omega, \rho)
    \end{align}
    for all $q > 1$.
    \item If $m$ satisfies $m > n\left(1 + \frac{1}{t}\right) - \frac{\gamma}{t}$ then we have 
    \begin{align}\label{R2weightedineq3}
    	\lVert v \rVert_{L^{\infty}(\Omega)} \le C\lVert \nabla_g v \Vert_{L^m(\Omega, \rho)}, \, \forall \, v \in W_0^{1,m}(\Omega, \rho).
    \end{align}
    \end{enumerate}
    Finally, we also have the following weighted Poincar\'{e} inequality 
    \begin{align}\label{R2weightedpoincare1}
    	\lVert v \rVert_{L^{m}(\Omega)} \le C\Lambda(\Omega) \lVert \nabla_g v \Vert_{L^m(\Omega, \rho)}, \, \forall \, v \in W_0^{1,m}(\Omega, \rho)
    \end{align}
    where, $\Lambda(\Omega) = |\Omega|^{\frac{1}{m} - \frac{1}{m^\sharp}}$.
    \\ The above inequalities are also valid for functions $v$ with zero mean, so that we have corresponding Poincar\'{e}-Sobolev and Poincar\'{e}-Wirtinger inequalities as well.
    \end{lemma}

    \begin{remark}\label{R2rmkweightedsobolev2}
    	The case $p > 2$ and $m = 2$ is a particularly relevant one. In this case we are in dimension $n > p > 2$. When $u$ is as in Lemma-\ref{R2W12reg} considering $\rho \equiv |\nabla_g u|_g^{p-2}$, from Lemma-\ref{R2negint1} we have $\rho$ satisfies the condition \eqref{R2negintcond1} for all $t = \frac{p-1}{p-2}r < \frac{p-1}{p-2} = 1 + \frac{1}{p-2}$ and $0 \le \gamma < n-2$. Then for the exponent $2^\sharp$ we have $$\frac{1}{2^\sharp} = \frac{1}{2} - \frac{1}{n} + \frac{n-\gamma}{2nt} > \frac{1}{2} - \frac{1}{n} + \frac{1}{n}\left(\frac{p-2}{p-1}\right) := \frac{1}{\overline{2}^\sharp}$$ where we used the fact that $n - \gamma > 2$ and $\frac{1}{t} > \frac{p-2}{p-1}$. Furthermore choosing $t > \frac{n-\gamma}{2}$ ensures $\frac{1}{2} > \frac{1}{2^\sharp}$. This would require $\frac{n-\gamma}{2} < t < \frac{p-1}{p-2}$, which is always possible if we choose $\gamma$ close to $n-2$ (recall that $\gamma < n-2$). Therefore for such choice of $t$ and $\gamma$ we have $2 < 2^\sharp < \overline{2}^\sharp$. This implies the weighted $2$-Sobolev Inequality \eqref{R2weightedsobolev1} with exponent $2^\sharp > 2$ and weighted $2$-Poincar\'{e} inequality \eqref{R2weightedpoincare1} (and weighted $2$-Poincar\'{e}-Wirtinger inequality) with weight $\rho = |\nabla_g u|_g^{p-2}$ are valid as well.
    \end{remark}
    \,
    \\ With the aid of the Sobolev and Poincar\'{e} inequalities one can employ a Moser's iteration argument exactly as in \cite{DS2} to get Harnack inequalities for bounded and non-negative sub(super)-solutions of the linearized equation \eqref{R2eqn45}. In case $u$ does not have any critical points then the Moser's iteration can be continued with the standard unweighted Sobolev and Poincar\'{e} inequalities.
    
    \begin{lemma}\label{R2harnack1}
    	Let $2< p < n$ and $v \in W^{1,2}(\Omega, \rho) \cap L^\infty(\Omega)$ be a non-negative weak supersolution to the linearized equation \eqref{R2eqn46} where, $\rho = |\nabla_g u|_g^{p-2}$ with $u$ being as in Lemma-\ref{R2W12reg}. Then for each $B_{2\delta}(x) \subset \Omega$ and any $s < \chi = \frac{2^\sharp}{2}$ there is a constant $C>0$ (depending on $x, s, n, p, u, f$) such that 
    	\begin{align}\label{R2h1}
    		\lVert v \rVert_{L^s(B_{2\delta}(x))} \le C \inf_{B_\delta(x)} v.
    	\end{align}
        \\ On the other hand if $v \in W^{1,2}(\Omega, \rho) \cap L^\infty(\Omega)$ be a non-negative weak subsolution to the linearized equation \eqref{R2eqn46}, then for each $B_{2\delta}(x) \subset \Omega$ and any $s < \chi = \frac{2^\sharp}{2}$ there is a constant $C>0$ (depending on $x, s, n, p, u, f$) such that 
        \begin{align}\label{R2h3}
        	 \sup_{B_\delta(x)} v \le C\lVert v \rVert_{L^s(B_{2\delta}(x))}.
        \end{align}
        If, $\frac{2n+2}{n+2} < p < 2$ then the same conclusions hold for $s < \chi' = \frac{2^\ast}{2}\left(1 - \frac{1}{q}\right)$ where, $q < \frac{p-1}{2-p}$ denotes the integrability of $\rho$ i.e., $\rho \in L^q(\Omega)$.
        \\ Thus if $v \in W^{1,2}(\Omega, \rho) \cap L^\infty(\Omega)$ be a non-negative weak solution to the linearized equation \eqref{R2eqn46}, then for each $B_{2\delta}(x) \subset \Omega$ we have
        \begin{align}\label{R2h4}
        	\sup_{B_\delta(x)} v \le C\inf_{B_\delta(x)} v.
        \end{align}
        \\ As a result we have the strong minimum-principle for non-negative solutions, i.e., either $v > 0$ in $\Omega$ or $v \equiv 0$ in $\Omega$. When $\C_u = \emptyset$ meaning when $u$ does not have any critical points then the requirement $p > \frac{2n+2}{n+2}$ can be replaced by any $p \in (1,n)$ and both the Harnack and consequently the strong minimum-principle remains valid.
    \end{lemma}
    \,
    \\ In the following lemma we state the strong comparison theorem which has been used in the proof of symmetry. The proof of the theorem again follows Moser's iteration scheme which can be continued with the standard unweighted Sobolev and Poincar\'{e} inequalities as long as we stay away from the critical point sets of the solutions in concern. 
    
    \begin{lemma}\label{R2strongcomparison1}
    	Let $u , v \in C^{1,\sigma}(\Omega)$ be solutions of the equation \eqref{R2eqn41}. Suppose, $u \le v$ in $\Omega$ and $\Omega' \subset \Omega \setminus (\C_u \cap \C_v)$ to be any connected component. Then for each $B_{2\delta}(x) \subset \Omega'$ we have the Harnack inequality 
    	\begin{align}\label{R2h2}
    		\sup_{B_\delta(x)} (v - u) \le C \inf_{B_\delta(x)} (v - u).
    	\end{align} for some constant $C > 0$ (depending on $x, s, n, p, u, v, f, \delta$).
    As a consequence, we have the strong-comparison theorem between $u$ and $v$ in each component $\Omega'$ of $\Omega \setminus (\C_u \cap \C_v)$ i.e., either $u \equiv v$ in $\Omega'$ or $u < v$ in $\Omega'$.
    \end{lemma}

\section{Existence and nonexistence} \label{R2exist}
In this section we will briefly look in to the question of existence and non existence of solution for the problem \eqref{R21eqn1}. In view of the symmetry result Theorem-\ref{R2hypsymmetry} it is enough to work in the subspace of $\mathcal{D}^{1,p}(\Hn)$ consisting of radial functions.\\\\
Fix a point $O\in \Hn$ and let $\mathcal{D}^{1,p}_{rad}(\Hn)$ be the closure of $C^1_c(\Hn)$ functions which are radial with respect to $O$. Let $u(x) = v(\dist_g(O,x))$ be a radial $C^1_c$ function, then for $t> 0$
$$
|v(t)| = \left|-\int_t^{\infty} v^\prime (s) \,ds \right| \le \left( \int_t^\infty|v^\prime (s) |^p  (\sinh s)^{n-1} \,ds\right)^{\frac{1}{p}} \left(\int_{t}^{\infty} \left(\sinh s\right)^ {-\frac{n-1}{p-1}}\,ds\right)^{\frac{p-1}{p}}.$$ Hence  for $R>0$ there exists a dimensional constant $C_R>0$ such that
\begin{equation}\label{R2R-bound}
|u(x)| \le C_R \|\nabla_g u\|_{L^p(\Hn \setminus B_R(O))} \; e^{-\frac{n-1}{p}\dist_g(O,x)} ,\; x \in \Hn \setminus B_R(O) 
\end{equation}
holds for all $u\in \mathcal{D}^{1,p}_{rad}(\Hn)$. As an immediate corollary we get
\begin{lemma}\label{R2radialcpt} 
	The embedding $\mathcal{D}^{1,p}_{rad}(\Hn) \hookrightarrow L^q(\Hn)$ is compact for any $p<q< p^\ast $.
\end{lemma}
Proof follows using standard arguments (see \cite{BS}, Theorem-3.1).

Using the above Compactness lemma we have the existence in the subcritical range $p < q < p^\ast$.
\begin{theorem} Let $p<q< p^\ast$, then the problem \eqref{R21eqn1} admits a solution.

\end{theorem}
\begin{proof} By standard arguments it is enough to show that $S_{\lambda , q}$ is achieved in \eqref{R2PSineq} by a function $u \in \mathcal{D}^{1,p}(\Hn)$. Note that by symmetrization arguments \cite[Chapter 7]{Baern}
$$ S_{\lambda , q} =\inf\limits_{u\in \mathcal{D}^{1,p}(\Hn) , u \not= 0} \frac{ \int_{\Hn} |\nabla_g u|_g^p - \lambda |u|^p \,dV_g }{ \left(\int_{\Hn} |u|^q \,dV_g\right)^{p/q}} = \inf\limits_{u\in \mathcal{D}_{rad}^{1,p}(\Hn) , u \not= 0} \frac{ \int_{\Hn} |\nabla_g u|_g^p - \lambda |u|^p \,dV_g }{ \left(\int_{\Hn} |u|^q \,dV_g\right)^{p/q}}.$$
Let, $u_n \in \mathcal{D}_{rad}^{1,p}(\Hn)$ be a minimizing sequence for the above infimum. With out loss of generality we may assume $u_n \ge 0$ and using the Ekeland variational principle, $u_n$ is a Palais-Smale sequence and $ \int_{\Hn} |\nabla_g u_n|_g^p - \lambda u_n^p \,dV_g =  \int_{\Hn} u_n^q \,dV_g$. Thus, $\int_{\Hn} |\nabla_g u_n|_g^p - \lambda u_n^p \,dV_g  \rightarrow \left(S_{\lambda , q}\right)^{\frac{q}{q-p}}$. Since $ \lambda < \left(\frac{n-1}{p}\right)^{p}$, this implies $u_n$ is bounded in $\mathcal{D}_{rad}^{1,p}(\Hn)$. Passing to a subsequence if necessary we may assume $u_n$ converges weakly to some $u\in \mathcal{D}_{rad}^{1,p}(\Hn)$.
 Passing to a further subsequence we may assume $u_n\rightarrow u$ in $L^q(\Hn)$ (using the previous Lemma-\ref{R2radialcpt}), pointwise a.e., $u_n^{p-1} \rightharpoonup u^{p-1}$ in $L^{\frac{p}{p-1}}(\Hn)$ and 
$|\nabla_g u_n|_g^{p-2} \nabla_gu_n \rightharpoonup V $ for some vector field $V \in L^{\frac{p}{p-1}}$. Since $u_n$ is a Palais-Smale sequence we have for for all 
$ \phi \in C_c^1(\Hn)$,
$$  \int_{\Hn} \left[|\nabla_g u_n|_g^{p-2}g(\nabla_gu_n,\nabla_g\phi) - \lambda u_n^{p-1} \phi \right] \,dV_g  = \int_{\Hn} u_n^{q-1} \phi \,dV_g + o(1), \; {\text as} \; n\rightarrow \infty.$$ 
Passing to the limit as $n\rightarrow \infty$ we see that 
$$\int_{\Hn} \left[g(V,\nabla_g\phi) - \lambda u^{p-1}\phi \right] \,dV_g  = \int_{\Hn} u^{q-1} \phi \,dV_g. $$ It follows from \cite{BM} that $V= |\nabla_g u|_g^{p-2}\nabla_gu$ and hence $u$ solves $-\Delta_p^{\Hn} u - \lambda u^{p-1} = u^{q - 1} \text{ in } \Hn$ and in particular $ \int_{\Hn} |\nabla_g u|_g^p - \lambda u^p \,dV_g =  \int_{\Hn} u^q \,dV_g$. Moreover $ \int_{\Hn} |u|^q \,dV_g = \lim\limits_{n\rightarrow \infty}\int_{\Hn} |u_n|^q \,dV_g = \left( S_{\lambda , q} \right)^{\frac{q}{q-p}} .$ Thus $u$ is a non-negative extremal.
\end{proof}
\,
\\ Next we consider the critical case $q= p^\ast$. The crucial step is to obtain higher integrability, see for example \cite{HD} for similar issues in the Euclidean context.
\begin{theorem}  Let $S_p$ be the best constant in the Euclidean Sobolev inequality defined by
$S_p = \inf\left\{ \int_{\R^n} |\nabla u|^p dx :  \int_{\R^n} |u|^{p^\ast} dx=1, u \in \mathcal{D}^{1,p}(\Rn)\right\}$. Then the problem \eqref{R21eqn1} admits a nontrivial solution if  $ S_{\lambda,p^\ast } < S_p$.
\end{theorem}
\begin{proof} We have the following Sobolev inequality  (see \cite{Aubin}, Theorem-2.28)
\begin{equation}\label{R2Hyp-Sob}
\int_{\Hn} |\nabla_g u|_g^p \ge S_{p} \left(\int_{\Hn} |u|^{p^\ast} \,dV_g\right)^{p/p^\ast}, \forall \, u \in \mathcal{D}^{1,p}(\Hn).
\end{equation}
Hence we have $S_{\lambda,p^\ast } \le S_p$. Choose $q_n \in (p, p^\ast)$ such that $ q_n \rightarrow p^\ast$ as $n\rightarrow \infty$. It follows from the previous theorem and our symmetry result that there exists a $u_n \in \mathcal{D}^{1,p}_{rad}(\Hn)$ satisfying 
$$\int_{\Hn} u_n^{q_n} \,dV_g = 1, \;  \int_{\Hn} \left[ |\nabla_g u_n|_g^p - \lambda u_n^p\right] \,dV_g =  S_{\lambda, q_n}$$
and $ -\Delta_p^{\Hn} u_n - \lambda u_n^{p-1} = S_{\lambda, q_n} u_n^{q - 1} \text{ in } \Hn$. \\\\
It is easy to see that $\limsup\limits_{n\rightarrow \infty} S_{\lambda, q_n} \le S_{\lambda, p^\ast}$, thus taking a subsequence if necessary we will assume that $\lim\limits_{n\rightarrow \infty} S_{\lambda, q_n} := \Lambda_0 \le S_{\lambda, p^\ast}$. This in particular implies that $u_n$ is bounded in $ \mathcal{D}^{1,p}_{rad}(\Hn)$. Therefore as in the proof of previous theorem we will be able to extract a subsequence, which we still denote by $u_n$, and a $u \in \mathcal{D}^{1,p}_{rad}(\Hn)$ such that 
$u_n\rightarrow u$ pointwise a.e., $u_n \rightharpoonup u$ in $\mathcal{D}^{1,p}_{rad}(\Hn)$ and $u$ solves $ -\Delta_p^{\Hn} u - \lambda u^{p-2} = \Lambda_0 u^{q - 1}  \text{ in } \Hn$. It remains to show that $u \not= 0$ and hence a constant multiple of $u$ solves \eqref{R21eqn1}.\\\\
Using the pointwise bound \eqref{R2R-bound} and dominated convergence theorem we can easily see that  $$\int_{\Hn \setminus B_R(O)} u_n^{q_n} \,dV_g \rightarrow \int_{\Hn \setminus B_R(O)} u^{p^\ast} \,dV_g$$ as $n\rightarrow \infty$ for any $R>0.$ Thus, $\int_{\Hn} u_n^{q_n} \,dV_g \rightarrow \int_{\Hn} u^{p^\ast} \,dV_g$ as $n\rightarrow \infty$ if we have for some $q> p^\ast$ and a $R>0$ such that $\int_{B_R(O)} u_n^{q} \,dV_g \; \le C , \forall \, n$. We will show that this is indeed true if we assume $S_{\lambda,p^\ast } <  S_p$.\\
Let us assume $S_{\lambda,p^\ast } <  S_p$ . Choose  $ \Lambda \in (\Lambda_0 ,S_p)$. Fix $R>0$ and  a radial function $\phi \in C_c^1(B_R(O))$ such that $\phi \equiv 1$ in $B_{R/2}(O)$. Let $k>1$ to be chosen later. Testing the PDE for $u_n$ with $\phi^pu_n^k $, we get
$$  \int_{\Hn} \left[|\nabla_g u_n|_g^{p-2}g(\nabla_g u_n,\nabla_g (\phi^p u_n^k)) - \lambda u_n^{p-1+k}\phi^p\right] \,dV_g = S_{\lambda, q_n} \int_{\Hn} u_n^{q_n-1+k}\phi^p \,dV_g.$$
Since $u_n$ is bounded in $\mathcal{D}^{1,p}(\Hn)$, the second term on the LHS is uniformly bounded if $p-1+k < p^\ast$.
Expanding the first term on LHS, we get
$$  k\int_{\Hn} |\nabla_g u_n|_g^{p}\phi^pu_n^{k-1}\,dV_g + \int_{\Hn} |\nabla_g u_n|_g^{p-2} g(\nabla_gu_n,\nabla_g\phi) (p\phi^{p-1} u_n^k)\,dV_g.$$
Since the integrand of the second term is supported in the annulus with radius $R/2$ and $R$, the radial bound \eqref{R2R-bound} implies that $u_n$ is uniformly bounded in this annulus. Thus the second term is bounded by $C\int_{B_R(O)}|\nabla_g u_n|_g^{p-1}\,dV_g$ which is uniformly bounded.  Now using Young's inequality we see that 
\begin{align*} 
	&\frac{kp^p}{(k+p-1)^p} \int_{\Hn} \left|\nabla_g \left(\phi u_n^{\frac{k+p-1}{p}}\right)\right|_g^{p} \,dV_g \\ &\le  k\int_{\Hn} |\nabla_g u_n|_g^{p}\phi^pu_n^{k-1}\,dV_g + C\int_{\Hn}\left(
 |\nabla_g \phi|_g^pu_n^{p+ k-1} + |\nabla_g u_n|_g^{p-1} u_n^k |\nabla_g \phi|_g \right) \,dV_g.\end{align*}
 The last integral is uniformly bounded as argued before if $k-1+p \le p^\ast$, thus substituting back these informations we get the existence of a constant $C>0$ such that
 $$\frac{kp^p}{(k+p-1)^p} \int_{\Hn} \left|\nabla_g \left(\phi u_n^{\frac{k+p-1}{p}}\right)\right|_g^{p} \,dV_g \le C + \Lambda \int_{\Hn} u_n^{q_n-1+k}\phi ^p\,dV_g.$$
 Now using \eqref{R2Hyp-Sob} the last term can be estimated as follows,
  \begin{align*} 
  	\Lambda \int_{\Hn} u_n^{q_n-1+k}\phi^p \,dV_g &= \Lambda \int_{\Hn} \left(u_n^{\frac{p-1+k}{p}}\phi\right)^p(u_n^{q_n-p}) \,dV_g \\ 
  	&\le \Lambda \left(\int_{\Hn} \left(u_n^{\frac{p-1+k}{p}}\phi\right)^{q_n}\,dV_g\right)^{\frac{p}{q_n}} \left(\int_{\Hn} u_n^{q_n} \,dV_g\right)^{\frac{q_n-p}{q_n}} \\
    &\le |V_g(B(O,R))|^{\frac{p^\ast-q_n}{p^\ast}} \Lambda \left(\int_{\Hn} \left(u_n^{\frac{p-1+k}{p}}\phi \right)^{p^\ast} \,dV_g\right)^{\frac{p}{p^\ast}} \\
    &\le  |V_g(B(O,R))|^{\frac{p^\ast-q_n}{p^\ast}} \frac{\Lambda}{S_p} \int_{\Hn} \left|\nabla_g \left(\phi u_n^{\frac{k+p-1}{p}}\right)\right|_g^{p} \,dV_g \\
    &\le (1-\alpha)\int_{\Hn} \left|\nabla_g \left(\phi u_n^{\frac{k+p-1}{p}}\right)\right|_g^{p} \,dV_g \end{align*}  
  as $q_n \to p^\ast$ for all $n$ large and some $\alpha \in (0,1)$. Choosing $k$ close to 1 such that $\frac{kp^p}{(k+p-1)^p} > 1-\frac{\alpha}{2}$  we get $ \int_{\Hn} \left|\nabla_g \left(\phi u_n^{\frac{k+p-1}{p}}\right)\right|_g^{p} \,dV_g \le C$ for all $n$ large and hence the higher integrability follows by Sobolev inequality \eqref{R2Hyp-Sob}. This completes the proof.
  \end{proof}
When $p=2$ the precise existence and nonexistence result was proved in \cite{MS}. When $p\not= 2$ we have the following existence theorem in the critical case.
 \begin{theorem} Let $1 < p < n$ such that $n > p^2$. When $ p < 2 $, $n > 3p-2$, $\lambda > 0$ and $q = p^\ast$ then the problem \eqref{R21eqn1} has a non zero solution.
 \end{theorem}
 Using standard test functions we can see that in this case $S_{\lambda , p^\ast} < S_p$, see the proof of Theorem-1.5 in \cite{Druet} for details.\\ 
  \,
  \\ We end the section with a Pohozhaev non-existence result for \eqref{R21eqn1} when $q = p^\ast$ and $\lambda = 0$.
\begin{theorem} Let $\lambda =0$ and $q= p^\ast$ then the problem \eqref{R21eqn1} does not have any non zero solution.
\end{theorem}
\begin{proof} Assume the problem has a nontrivial solution. Then by our symmetry result Theorem-\ref{R2hypsymmetry} the initial value problem
\begin{align}
\begin{cases}\label{R2pode}
\left( (\sinh t)^{n-1}|u^\prime|^{p-2} u^\prime\right)^\prime + (\sinh t)^{n-1} u^{p^\ast-1} = 0 \\
u(0)=\alpha , u^\prime(0)=0 , u(t)>0,  u^\prime(t)< 0 , \forall \, t > 0
\end{cases} 
\end{align}  
has a solution for some $\alpha >0$ satisfying the asymptotic estimates stated in Theorem-\ref{R2hypsymmetry} namely $u(t) \sim u^{\prime}(t) \sim e^{-\frac{n-1}{p-1}t}$ for $t$ large.\\
Multiplying the equation \eqref{R2pode} with $(\sinh t)u^\prime (t)$ and integrating from $0$ to some $R>0$, we get
\begin{align} 0 &= \int_0^R \left[ \left( (\sinh t)^{n-1}|u^\prime|^{p-2} u^\prime\right)^\prime + u^{p^\ast-1} (\sinh t)^{n-1}\right](\sinh t)u^\prime (t) \; dt \nn \\
& = - \int_0^R  \left( (\sinh t)^{n-1}|u^\prime|^{p-2} u^\prime\right)((\sinh t)u^\prime (t))^\prime\,dt \nn \\ 
& \quad + \frac{1}{p^\ast} \int_0^R
\left(u^{p^\ast}\right)^\prime(\sinh t)^n\; dt + (\sinh R)^n|u^\prime(R)|^p.
\end{align}
Again integrating by parts and rearranging the terms, we get
\begin{equation}\label{R2Poh-1}
\frac{p-n}{p}\int_0^R\left( |u^\prime|^p - u^{p^\ast}\right)(\sinh t)^{n-1}\cosh t \; dt = \left[\frac{p-1}{p} |u^\prime (R)|^p + \frac{1}{p^\ast}u^{p^\ast} (R)\right] (\sinh R)^n 
\end{equation}
Similarly multiplying the equation \eqref{R2pode} by $(\cosh t) u(t)$ and simplifying as before we get
\begin{align}\label{R2Poh-2}
\int_0^R \left( |u(t)|^{p^\ast} - |u^\prime (t)|^p\right) (\sinh t)^{n-1}\cosh t \; dt &= \int_0^R |u^\prime (t)|^{p-2}  u^\prime (t) u(t)(\sinh t)^{n} \; dt \nonumber \\ & \quad + |u^\prime (R)|^{p-1} u(R)  (\sinh R)^{n-1}\cosh R.
\end{align}

Substituting \eqref{R2Poh-2} in \eqref{R2Poh-1} and simplifying using the asymptotic estimates of $u$ and $u'$ we get,
$$ \frac{n-p}{p}\int_0^R|u^\prime (t)|^{p-2} u(t) u^\prime(t) (\sinh t)^{n} \; dt = O\left(e^{\left( n- \frac{n-1}{p-1} p\right) R } \right) + O\left(e^{\left( n- \frac{n-1}{p-1} p^\ast\right) R} \right) \to 0$$ as $R\rightarrow \infty$.
Thus we get $\int_0^\infty |u^\prime (t)|^{p-1} u(t)(\sinh t)^{n} \; dt= 0$, which is a contradiction.
\end{proof}

\,
\\ \textit{Acknowledgments.} This work was done as a part of the PhD thesis of the first author. The work was supported by the Department of Atomic Energy, project ``Mathematics, Theoretical Sciences and Science Educations" Project Identification Code RTI4001.

	\bibliographystyle{plain}
\bibliography{RamyaSan.bib}

\end{document}